\documentclass[11pt,oneside,reqno]{amsart}
\usepackage{amsmath, amssymb, amsthm}
\usepackage{url}
\usepackage{mathrsfs}
\usepackage[ansinew]{inputenc}
\usepackage[breaklinks]{hyperref}
\usepackage{multirow}

\allowdisplaybreaks

\renewcommand{\(}{\left\(}
\renewcommand{\)}{\right\)}
\renewcommand{\[}{\left\[}
\renewcommand{\]}{\right\]}
\newtheorem{remark}[]{Remark}
\numberwithin{equation}{section}
 \theoremstyle{plain}
\newtheorem{theorem}{Theorem}[section]
\newtheorem{lemma}[theorem]{Lemma}
\newcommand{\sm}{\left(\begin{smallmatrix}}
\newcommand{\esm}{\end{smallmatrix}\right)}

\newtheorem{corollary}[theorem]{Corollary}
\newtheorem{proposition}[theorem]{Proposition}

	\newtheorem{definition}[theorem]{Definition}
	
   \makeatletter
\def\proof{\@ifnextchar[{\@oproof}{\@nproof}}
\def\@oproof[#1][#2]{\trivlist\item[\hskip\labelsep\textit{#2 \textbf{Proof of}\
#1.}~]\ignorespaces}
\def\@nproof{\trivlist\item[\hskip\labelsep\textit{Proof.}~]\ignorespaces}

\makeatother

\usepackage{color}
\usepackage{amsmath}

\definecolor{blue}{rgb}{0,0,1}
\definecolor{red}{rgb}{1,0,0}
\definecolor{green}{rgb}{0,.6,.2}
\definecolor{purple}{rgb}{1,0,1}

\long\def\red#1\endred{{\color{red}#1}}
\long\def\blue#1\endblue{{\color{blue}#1}}
\long\def\purple#1\endpurple{{\color{purple}#1}}
\long\def\green#1\endgreen{{\color{green}#1}}

\begin{document}
\title[Period Function   of Maass forms    from Ramanujan's Lost Notebook]{Period Function                                                                                                                                                                                                                                                                       of Maass forms   from Ramanujan's Lost Notebook}
\author{YoungJu Choie}
\address{Department of Mathematics, Pohang Universtity of Science and Technology, \\
Pohang, Republic of Korea}
\email{yjc@postech.ac.kr}
\author{Rahul Kumar}
\address{Department of Mathematics, Indian Institute of Technology, Roorkee-247667, Uttarakhand, India}
\email{rahul.kumar@ma.iitr.ac.in} 


\subjclass[2020]{Primary 11F67, 11F12,  11F99 Secondary 30D05, 33E20}
  \keywords{Period functions of Maass forms, Hecke Eigenform, Lost Notebook, Kronecker limit formula, functional equations, Herglotz function}
\maketitle
\pagenumbering{arabic}
\pagestyle{headings}
\begin{abstract}
The Lost Notebook of Ramanujan contains a number of beautiful formulas, one of which can be found on its page 220. It involves an interesting function, which we denote as $\mathcal{F}_1(x)$.  In this paper, we show that $\mathcal{F}_1(x)$ belongs to the category of period functions as it satisfies the period relations of Maass forms in the sense of Lewis and Zagier \cite{lz}. Hence, we refer to $\mathcal{F}_1(x)$ as the \emph{Ramanujan period function}.  Moreover, one of the salient aspects of the Ramanujan period function $\mathcal{F}_1(x)$ that we found out is that it is a Hecke eigenfunction under the action of Hecke operators on the space of periods.  We also establish that it naturally appears in a Kronecker limit formula of a certain zeta function, revealing its  connections to various topics. Finally, we generalize   $\mathcal{F}_1(x)$  to include a parameter $s,$  connecting our work to the broader theory of period functions developed by Bettin and Conrey \cite{bc} and Lewis and Zagier \cite{lz}. We emphasize that Ramanujan was the first to study this function, marking the beginning of the study of period functions. 
 
\end{abstract}

\maketitle

\tableofcontents
\section{Introduction and Motivation}

Consider the function
\begin{align}\label{refined ramanujan function}
\mathcal{F}_1(x):=\mathscr{F}_1(x)-\frac{1}{2}\left(\gamma-\log\left(\frac{2\pi}{x}\right)\right), \quad x\in\mathbb{C}\backslash(-\infty,0],
\end{align}
where $\gamma $ is the Euler's constant and
$\mathscr{F}_1(x)$ is defined as  
\begin{align}\label{ramanujan function}
\mathscr{F}_1(x):=\sum_{n=1}^\infty\left\{\psi(nx)+\frac{1}{2nx}- \log{(nx)}\right\}, \quad x\in\mathbb{C}\backslash(-\infty,0] 
\end{align}
where  $\psi(x)$ is the  digamma function $\psi(x):=\Gamma'(x)/\Gamma(x)$, with $\Gamma(x)$ being the Euler Gamma function.

On page 220 of his Lost Notebook \cite{rlnb}, Ramanujan recorded the following beautiful symmetric result involving $\mathscr{F}_1(x)$ \cite[p.~293, Entry 13.3.1]{andrewsberndt}
\begin{align}
\sqrt{x} \left\{
\frac{\gamma-\log{(2\pi x)}}{2x}+\mathscr{F}_1(x)\right\}
&=\sqrt{1/x} \left\{
\frac{\gamma-\log{(2\pi/ x)}}{2/x}+\mathscr{F}_1\left(\frac{1}{x}\right)\right\}\label{ramanujan fe}\\
&=-\pi^{-\frac{3}{2}}\int_0^{\infty} \left|\Xi\left(\frac{t}{2}\right)\Gamma\left(\frac{-1+it}{4}\right)\right|^2
 \frac{\cos{ (\frac{1}{2}t\log{x})  }}
{1+t^2}dt,\nonumber
\end{align}
where $\Xi(t)$ denotes Riemann's function defined by \cite[p.~16, Equations (2.1.14), (2.1.12)]{titch}
\begin{align*}
\Xi(t):=\xi(\tfrac{1}{2}+it)\quad \mathrm{and}\quad \xi(s):=\frac{s}{2}(s-1) \pi^{-\frac{s}{2} }\Gamma(\tfrac{s}{2} )\zeta(s),
\end{align*}
where $\zeta(s)$ is the Riemann zeta function. A proof of the above result was given by Berndt and Dixit \cite{BD} for the first time.
 
In Section \ref{rel} we show  that the function $\mathcal{F}_1(x)$
satisfies period relations of Maass forms in the sense of Lewis-Zagier \cite{lz} with spectral parameter $s=\frac{1}{2},$  that is, it satisfies   a three-term functional equation: 
\begin{align*} 
\mathcal{F}_1(x)-\mathcal{F}_1(x+1)- \frac{1}{x+1}\mathcal{F}_1\left(\frac{x}{x+1}\right)=0.
\end{align*}

We will henceforth refer to  the function $\mathcal{F}_1(x)$ as the \emph{``Ramanujan period function"}.

One of the  main results of this paper is that the function $\mathcal{F}_1(x)$ is a Hecke eigenform:
\begin{align*}
(\mathcal{F}_1|_1\widetilde{T}_n)(x)=\sqrt{n}\ d(n)\mathcal{F}_1(x),
\end{align*}
where $d(n)$ is the  usual divisor function, $\widetilde{T}_n$ represents   Hecke operators acting on the space of periods (\cite{choiezagier}, \cite{raza}). For more details, see Section \ref{hecke operator section} below. 

 We also show that the Ramanujan period function $\mathcal{F}_1(x)$ naturally appears as a Kronecker limit formula for a certain partial zeta function over real quadratic fields   (see Section \ref{kronecker section}).  
In Section \ref{higher kronecker section}, we establish higher Kronecker limit formulas for this zeta function. We study the arithmetic properties of a function $J_1(x), $ which is a combination of $\mathcal{F}_1(x)$. The evaluation of $J_1(x)$ at units of real quadratic fields is also provided as an application of the Kronecker limit formula. Further, using the two-term functional equation of  $\mathcal{F}_1(x) $, we find a connection between the Ramanujan period function  $\mathcal{F}_1(x)$ and Eisenstein series of weight $1$ in Section \ref{connectionEisen}.

In Section \ref{generalized RPF}, we discuss how the function $\mathcal{F}_1(x)$ can be extended to include a general complex  parameter $s$. Interestingly, this generalization has appeared in the works of Lewis and Zagier \cite{lz} as well as Bettin and Conrey \cite{bc}, albeit in different forms. However, neither of these papers refers to the Ramanujan's contributions. Ramanujan was, in fact, the first person to study the function $\mathcal{F}_1(x)$, and his work on this function can be seen as the initial step in the broader theory of period functions later studied by Lewis and Zagier \cite{lz} and Bettin and Conrey \cite{bc}.

\section{Period relations }\label{rel}
Our first result of this section shows that the Ramanujan period function $\mathcal{F}_1(x)$ falls under the category of period functions of Maass cusp forms in the sense of Lewis-Zagier \cite{lz}.

\begin{theorem}\label{cleaner functional equations}
 {\textbf{ \textup{(}Period relations\textup{)}} } Let $x\in\mathbb{C}\backslash(-\infty,0]$. 
\begin{enumerate}
\item  The Ramanujan period function $\mathcal{F}_1(x)$ satisfies the period relation:
\begin{align}\label{refined 3-term}
\mathcal{F}_1(x)-\mathcal{F}_1(x+1)-\frac{1}{x+1}\mathcal{F}_1\left(\frac{x}{x+1	}\right)=0.
\end{align}
\item Moreover, it has the following symmetric relation:
\begin{align}\label{refined 2-term}
\mathcal{F}_1(x)-\frac{1}{x}\mathcal{F}_1\left(\frac{1}{x}\right)=0.
\end{align}
\end{enumerate}
 \end{theorem}

 Note that the analyticity of the Ramanujan period function $\mathcal{F}_1(x)$ as a function of $x$ in the region $\mathbb{C}\backslash(-\infty, 0]$ is evident owing to the fact that \cite[p.~259, formula 6.3.18]{as}
\begin{align*}
\psi(x)-\log(x)+\frac{1}{2x}=\mathcal{O}\left(\frac{1}{x^2}\right),\ \mathrm{as}\ x\to\infty\ \mathrm{in}\ |\arg(x)|<\pi.
\end{align*}


 Our next result provides asymptotics for $\mathcal{F}_1(x)$ which agree with Lewis and Zagier \cite[p.~236--237]{lz} in the case of $s=1/2.$  
\begin{theorem}\label{asymptotics}
As $x\to\infty$ in $|\arg(x)|<\pi$, we have
\begin{align}\label{xtendinfinity}
\mathcal{F}_1(x)\sim-\frac{1}{2}\left(\gamma-\log\left(\frac{2\pi}{x}\right)\right)+\sum_{n=2}^\infty\frac{\zeta(1-n)\zeta(n)}{x^n},
\end{align}
and as $x\to0$ in $|\arg(x)|<\pi$,
\begin{align}\label{xtendzero}
\mathcal{F}_1(x)\sim-\frac{1}{2x}\left(\gamma-\log\left(2\pi x\right)\right)+\sum_{n=2}^\infty\zeta(1-n)\zeta(n)x^{n-1}.
\end{align}
\end{theorem}

\begin{remark}
Bettin and Conrey \cite{bc} as well as Lewis and Zagier derived the three-term functional equation given above in \eqref{refined 3-term}. However, we offer a new proof of \eqref{refined 3-term} that does not rely on the periodicity of the Eisenstein series. Moreover, this paper goes beyond these results by exploring further properties of the function $\mathcal{F}_1(x)$, such as the Hecke action, its role in the Kronecker limit formula, and other applications, which are discussed in the upcoming sections.
\end{remark}


\medskip
  
  
\section{Hecke Eigenform}\label{hecke operator section}

In this section we follow the notations from Choie-Zagier\cite{choiezagier} or Radchenko and Zagier \cite{raza}.   
Let $\Gamma$ be the group PSL$_2(\mathbb{Z}),$ which is generated by the matrices $S=\begin{pmatrix}
  0 & -1\\ 1 & 0   
\end{pmatrix}$ and $U=\begin{pmatrix}
    1 & -1\\ 1& 0 
\end{pmatrix}$ with
$S^2=U^3=1$, and let $T=US=\begin{pmatrix}
    1 & 1 \\ 0 & 1
\end{pmatrix}.$
Let $$\mathcal{M}_n:=\{ A\in M_{2\times 2}(\mathbb{Z}) / \{\pm I\}  :\, det(A)=n\},\qquad n\in \mathbb{N}.$$ Further let $\mathcal{R}_n=\mathbb{Q}[\mathcal{M}_n].$ 

Recall that  the Hecke operator $\widetilde{T}_n \in \mathcal{R}_n,$ for each $n\in \mathbb{N},$ acts   on the space of  period functions 
if it satisfies the following relation (see\cite{choiezagier, raza}): 
\begin{eqnarray}\label{existence}
(1-S)\widetilde{T}_n=T_n^{\infty}(1-S)+(1-T)Y,
\end{eqnarray}
for some $Y\in \mathcal{R}_n, $ where $ T_n^{\infty}$ is the usual representative of the $n$-th Hecke operator on the space of modular forms \cite{choiezagier}
$$T_n^{\infty}=\sum_{ad=n}\sum_{0\leq b<d}\begin{bmatrix}
a & b \\
0 & d 
\end{bmatrix} \in \mathcal{R}_n.$$

Recall the standard slash operator of weight $k$ defined as
\begin{align*}
(f|_k\gamma)(x)=\frac{(ad-bc)^{k/2}}{(cx+d)^k}f\left(\frac{ax+b}{cx+d}\right),
\end{align*}
with $\gamma$ being the matrix $\gamma=\begin{pmatrix}
a & b\\ c& d \end{pmatrix}.$

We are now ready to state our main result of this section which shows that the Ramanujan period function $\mathcal{F}_1(x)$ is a Hecke eigenfunction under the action of $\widetilde{T}_n$.
\begin{theorem}\label{hecke operator functional equations}
	Let $\widetilde{T}_n=\sum_{\gamma}v_\gamma\gamma$  represent the $n$-th Hecke operator on periods and that all matrices in $\widetilde{T}_n$ have non-negative entries. Let $x>0$. Then
the Ramanujan period function $\mathcal{F}_1(x)$ is a Hecke eigenfunction of the Hecke operator $\widetilde{T}_n$\textup{:}
\begin{align*}
(\mathcal{F}_1|_1\widetilde{T}_n)(x)=\sqrt{n}\ d(n)\mathcal{F}_1(x),
\end{align*}
where $d(n)$ is the divisor function which counts the total number of the divisors of $n$.
\end{theorem}

We next provide a special case of the above theorem. To that end, let us define 
\begin{align}\label{tntilde}
\widehat{T}_n:=\sum_{\substack{0\leq c<a\\0\leq b<d\\ad-bc=n}}\begin{bmatrix}
a&b\\
c&d
\end{bmatrix}.
\end{align}
Then \cite[Proposition 3]{raza} shows that  $\widehat{T}_n\in\mathcal{R}_n$ acts like the $n$-th Hecke operator on periods. More precisely, there exists an element $Y_n\in\mathcal{R}_n$ such that \eqref{existence} is satisfied.

Applying Theorem \ref{hecke operator functional equations} with $\widetilde{T}_n=\widehat{T}_n$ yields the following corollary.
\begin{corollary}\label{hecke operator functional equations corollary}
Let $x>0$ and $n\in\mathbb{N}$. Then the Ramanujan period function $\mathcal{F}_1(x)$ is a Hecke eigenfunction under $\widehat{T}_n :$
\begin{align*}
(\mathcal{F}_1|_1\widehat{T}_n)(x)=\sqrt{n}\ d(n)\mathcal{F}_1(x).
\end{align*}
\end{corollary}

 For instance, when $n=2$ in \eqref{tntilde} then  $\widehat{T}_n$   is
\begin{align*}
\widehat{T}_2&=\begin{bmatrix}
1&0\\
0&2
\end{bmatrix}+\begin{bmatrix}
1&1\\
0&2
\end{bmatrix}+\begin{bmatrix}
2&0\\
0&1
\end{bmatrix}+\begin{bmatrix}
2&0\\
1&1
\end{bmatrix},
\end{align*}
 hence we get 
\begin{align}\label{5-term}
2\mathcal{F}_1\left(x\right)- \mathcal{F}_1\left(2x\right)-\frac{1}{2}\mathcal{F}_1\left(\frac{x}{2}\right)-\frac{1}{2}\mathcal{F}_1\left(\frac{x+1}{2}\right)-\frac{1}{x+1}\mathcal{F}_1\left(\frac{2x}{x+1}\right)=0.
\end{align}

We conclude this section by noting that  one can find a relation of the type given in \eqref{5-term} for every $n\in\mathbb{N}$ and $\widehat{T}_n$ through application of Corollary \ref{hecke operator functional equations corollary}. 
\medskip

\section{Arithmetic Applications}
\subsection{ Kronecker's limit formula  }\label{kronecker section}


Let K be the real quadratic fields with discriminant $D > 0$. Recall the relation between ideal classes   $\mathscr{B}$ of $K$ and reduced numbers (page 163 in \cite{zagier1975}):
a reduced number $w\in K $ satisfies
\begin{eqnarray}\label{red}
   w>1 > w' >0 
\end{eqnarray} 
and  has a pure periodic continued fraction expansion. 
 If  $\mathscr{B}$ is a narrow ideal class of $K$ with length  $r=\ell(\mathscr{B})$  and cycle $((b_1, \cdots, b_r)) ( b_i\in \mathbb{Z}, b_i \geq 2),$ there are exactly $r$ many reduced numbers $w\in K$ for which $\{1,w\}$ is a  basis for some ideal in $\mathscr{B},$ say the numbers
 
\begin{eqnarray} \label{redu}
w_k =  b_k-\cfrac{1}{b_{k+1}-\ \genfrac{}{}{0pt}{0}{}{\ddots\genfrac{}{}{0pt}{0}{}\ \genfrac{}{}{0pt}{0}{}{-\cfrac{1}{b_{r}-\cfrac{1}{b_{ 1}-\genfrac{}{}{0pt}{0}{}{\ddots}}}}}} 
\end{eqnarray}
for $k=1, \cdots, r.$  Let Red$(\mathscr{B}):=\{w_1, w_2, \cdots, w_r \}$
be the set of all reduced forms in $\mathscr{B}.$  We extend the definition of $w_k$ to $k\in\mathbb{Z}$, where $w_k$ depends only on the class of 
$k\pmod{r}.$
Fix  the ideal $\underline{b}=\mathbb{Z}w_0+\mathbb{Z} \in \mathscr{B}^{-1}.$  Consider the following quadratic form
\begin{eqnarray*}
Q_k(p,q)=\frac{1}{w_k-w_k'}(q+pw_k)(q+pw_k'), 1\leq k \leq r.
\end{eqnarray*} 

Now let us define the following zeta function associated with real quadratic fields: 
\begin{definition}
  Define  the following  partial zeta function  
\begin{eqnarray}\label{qdr}
D^{k/2}\ \widetilde{\zeta}(s, \mathscr{B} )
: =
\sum_{k=1}^r\sum_{p\geq 1,q\geq0}
 \frac{p^s}{Q_k(p,q)^s} .
\end{eqnarray}
\end{definition}

\begin{remark}
Note that the zeta function in \eqref{qdr} is a modified version of the  Dedekind zeta function 
\begin{eqnarray}\label{de}
D^{k/2} {\zeta}(s, \mathscr{B} )
=
\sum_{k=1}^r\sum_{p\geq 1,q\geq0}
 \frac{1}{Q_k(p,q)^s} 
 \end{eqnarray}
 of the narrow ideal class $\mathscr{B}$ of the real quadratic field $K=\mathbb{Q}(\sqrt{D})$. 
In his breakthrough work \cite{zagier1975},  Zagier obtained an explicit formula for the Kronecker's  limit formula of ${\zeta}(s, \mathscr{B} ).$ 
\end{remark}


\begin{theorem}\label{analyticity>2}{\bf (Kronecker's limit formula)}\\
Consider the indefinite binary quadratic form $Q(x,y)=ax^2+bxy+cy^2,$ where $ a,b,c>0$ and the discriminant $b^2-4ac=1$.  The roots of the associated quadratic equation $cw^2-bw+a=0$, denoted as $w$ and $w'$, are ordered such that $w>w'>0$. Using these roots, the quadratic form $Q(x,y)$ can be expressed as $Q(x,y)=\frac{1}{w-w'}(y+xw)(y+xw')$. \\
Let us define the function
\begin{align}\label{Z1}
\mathcal{Z}(s; w, w'):=\sum_{p\geq1,q\geq0} \left(\frac{p (w-w')}{(q+pw)(q+pw')}\right)^s.
\end{align}
Then, we have
\begin{enumerate}
\item
\begin{align}\label{re}
D^{k/2}\ \widetilde{\zeta}(s,\mathscr{B}) =\sum_{k=1}^r\mathcal{Z}(s; w_k, w_k'),
\end{align}
where $w_{k}, i\leq k\leq r$ is the reduced number in (\ref{redu}).
\item  $ \mathcal{Z}(s; w, w')$ is defined in the region $\mathrm{Re}(s)> 2$.
 
\item  $\mathcal{Z}(s; w, w')$  has an analytic continuation in the plane $\mathrm{Re}(s)>\frac{1}{2}$ with having only two simple poles at $s=1$ and $s=2$.

\item the Laurent series expansion of $\mathcal{Z}(s;w,w')$ around $s=1$ is given by
\begin{align}\label{first kronecker limit formula eqn}
   \mathcal{Z}(s; w, w') =\frac{\frac{w-w'}{2ww'}}{s-1}+\widetilde{P}(w,w')+\mathcal{O}(s-1),
\end{align}
where 
\begin{align}\label{p1 in terms of f}
\widetilde{P}(x,y)&:=\mathcal{F}_1(x)-\mathcal{F}_1(y)-\frac{x-y}{2xy}\left(\gamma+\log\left(\frac{x-y}{xy}\right)\right).
\end{align}
\end{enumerate}
\end{theorem}

\bigskip

 \subsection{Higher Kronecker limit formula}\label{higher kronecker section}

Let us define the following function:
\begin{definition}\label{fk}
For $x\in\mathbb{C}\backslash(-\infty,0]$, we define 
\begin{align*}
\mathscr{F}_k(x):=
\begin{cases}
\displaystyle\sum_{n\geq 1} \frac{1}{n^{k-1} }\psi(nx), & \mathrm{if}\ k\geq 3,\\
\displaystyle\sum_{n\geq 1}\frac{1}{n}\big( \psi(nx)-\log{(nx)}\big), &  \mathrm{if}\ k=2,\\
\displaystyle\sum_{n\geq 1}\left(\psi(nx)-\log{(nx)}+\frac{1}{2 nx}\right), & \mathrm{if}\ k=1.
\end{cases}
\end{align*}
 \end{definition}

\begin{remark}
Note that, in the above definition, the instance where $k=2$ corresponds to the Herglotz function studied by Zagier \cite{zagier1975}, while the case when $k=1$ gives the Ramanujan function in \eqref{ramanujan function}. This explains the subscript $1$ in the notation of the function $\mathscr{F}_1(x)$ in \eqref{ramanujan function} or Ramanujan period function $\mathcal{F}_1(x)$ in \eqref{refined ramanujan function}.
\end{remark}

For $k\geq3$ and $k$ being an even natural number, the function $\mathscr{F}_k(x )$ appears in the higher Kronecker limit formula of the Dedekind zeta function defined in \eqref{de} (see \cite{vz}). 
To be specific,  Vlasenko and Zagier in \cite{vz}  established that for any narrow ideal class $\mathscr{B}\in\mathrm{Cl}^+(K)$ and integer $k\geq2$ \cite[Theorem 2]{vz}
\begin{eqnarray}\label{vz higher}
D^{k/2}\zeta(k, \mathscr{B} )
 =\sum_{w\in\mathrm{Red}(\mathscr{B})}(\mathscr{D}_{k-1}\mathscr{F}_{2k})(w,w'), 
\end{eqnarray}
where  $\mathscr{D}_{n}$  is a differential operator defined as \cite[p.~33, Definition 6]{vz}\textup{:}
\begin{align}\label{doperator}
(\mathscr{D}_{n}F)(x,y)=\sum_{i=0}^n\binom{2n-i}{n}\frac{F^{(i)}(x)-(-1)^iF^{(i)}(y)}{i!(y-x)^{n-i}},
\end{align}
for $ n\in\mathbb{N}\cup\{0\}$ and any function $F$. For notations and more details, we refer the reader to the paper of Vlasenko and Zagier \cite{vz}.

After observing that the function $\mathscr{F}_{k}(x)$ appears in \eqref{vz higher} only with even values of $k\geq4$, a natural question arises: \emph{what about the odd values of $k$\textup{?} More precisely, does there exist a zeta function for which either $\mathscr{F}_{2k+1}(x)$ or $\mathscr{F}_{k}(x)$, for all $k\in\mathbb{N}$ with $k\geq3$, plays a pivotal role in its higher Kronecker limit formula}? We answer this question affirmatively in the next theorem and show that the anticipated zeta function is nothing but that we have already defined in \eqref{qdr} above.

\begin{theorem}\label{new limit formula}{\bf (higher Kronecker limit formula)} Let  $\widetilde{\zeta}( k,  \mathscr{B})$ be defined in \eqref{qdr}. Let the differential operator $\mathscr{D}_{n}$ be defined in \eqref{doperator}. Then  for  any natural number $k\geq 3,$ we have
\begin{align}\label{limit formula}
 D^{\frac{k}{2}}\ \widetilde{\zeta}( k,  \mathscr{B})&= \sum_{w\in \mathrm{Red}(\mathscr{B})}  (\mathscr{D}_{k-1}\mathscr{F}_{k})(w,w').
\end{align}
 \end{theorem}
A table containing the numerical verification of the above result is given in the next subsection.


\subsection{Tables containing the numerical verification}\label{tables}
The number field $K =\mathbb{Q}(\sqrt{3})$ has two narrow classes $\mathscr{B}_0$ and $\mathscr{B}_1$ with the corresponding sets of reduced quadratic irrationalities being \cite[p.~36]{vz}
\begin{align*}
\mathrm{Red}(\mathscr{B}_0)=\{2+\sqrt{3}\} \quad \mathrm{and}\quad \mathrm{Red}(\mathscr{B}_1)=\left\{1+\frac{1}{\sqrt{3}},\frac{3+\sqrt{3}}{2}\right\}.
\end{align*}

To evaluate the term $D^{k/2}\ \widetilde{\zeta}(k,\mathscr{B})$ in the following table, we made use of the right-hand side of \eqref{re}.

\begin{center}
\begin{tabular}{ |p{1.5cm}|p{0.7cm}|p{3.9cm}|p{3.9cm}|}
 \hline
\begin{center} $\mathrm{Red}(\mathscr{B})$ \end{center} & \begin{center}$k$ \end{center}& \begin{center}$D^{k/2}\ \widetilde{\zeta}(k,\mathscr{B})$\end{center} &  \begin{center}RHS of formula \eqref{limit formula}\end{center}\\
 \hline
 \hspace{2mm}\begin{center}$\mathrm{Red}(\mathscr{B}_0)$\newline$=\{2+\sqrt{3}\}$   \end{center}&  \vspace{1.6mm} \begin{center}3 \end{center} \vspace{1mm}\begin{center}\vspace{2mm}4\end{center}   & \begin{center}\vspace{-1mm} $51.304025670384526\cdots$\end{center} \begin{center}\vspace{3mm}$156.2731732710374\cdots$ \end{center}& \vspace{0.5mm}\begin{center}\vspace{-1mm}$51.304025667471024\cdots$\end{center}\begin{center}\vspace{-3mm}$156.27317327706047\cdots$\end{center}\\ 
\hline
\begin{center}$\mathrm{Red}(\mathscr{B}_1)$\newline$=\{1+\frac{1}{\sqrt{3}},\newline \frac{3+\sqrt{3}}{2}\}$\end{center}   & \vspace{0.5mm} \begin{center}3 \end{center} \begin{center}\vspace{5.5mm}4\end{center}  & \vspace{0.5mm}\begin{center}$8.46173083386907\cdots$\end{center}\begin{center}$11.729190698921457\cdots$ \end{center}& \vspace{0.5mm}\begin{center}$8.461730833267936\cdots$\end{center}\begin{center}$11.729190666820669\cdots$\end{center}\\ 
\hline
\end{tabular}
\end{center}


 \subsection{Evaluation at units}\label{j1 section}
Let us define the function $J(x)$ as \cite[p. 145 \& p. 271]{cohen}
\begin{align*}
J(x):=\int_0^\infty\frac{\log\left(1+e^{-xt}\right)}{1+e^t}dt, \qquad \mathrm{Re}(x)>0.
\end{align*}
This function has gained significant interest and has been studied by various authors over time. Importantly, $J(x)$ appeared in the works of Herglotz \cite{her} and Muzaffar and Williams \cite{muzwil} in connection to various Kronecker limit formulas.
Recently, $J(x)$ has been extensively studied by Radchenko and Zagier \cite{raza}. For example, they found its special values when $x$ is a quadratic unit of a real quadratic field by employing Zagier's Kronecker limit formula from \cite{zagier1975}. Very recently, the current authors \cite{choiekumar} have further developed the theory of the function $J(x)$, providing special values and two-term and three-term functional equations for it.

The function $J(x)$ has a connection with the Herglotz function defined in (\ref{fk}). $$\mathscr{F}_2(x) =\sum_{n\geq 1} \frac{\psi(nx)-\log{(nx)}}{n}, \quad x\in \mathbb{C} \backslash (-\infty, 0], $$  as shown in \cite[p.~230]{raza}
\begin{align}\label{jxfx}
J(x)=\mathscr{F}_2(2x)-2\mathscr{F}_2(x)+ \mathscr{F}_2\left(\frac{x}{2}\right)+\frac{\pi^2}{12x}.
\end{align}

In this section, we introduce an analogue of $J(x)$ denoted by $J_1(x)$. It is defined as follows:
\begin{align}\label{j1x}
J_1(x):&=\int_0^\infty\frac{dt}{\left(1+e^{-t}\right)\left(1+e^{xt}\right)}, \qquad \mathrm{Re}(x)>0.
\end{align}

Further let
$$\mathcal{J}_1(x):=J_1(x)+\log(2).$$

Our next result establishes a connection between $J_1(x)$ and $\mathscr{F}_1(x)$ analogous to the relation given in \eqref{jxfx} along with a two-term functional equation for the function $\mathcal{J}_1(x)$.
\begin{theorem}\label{connection}
Let   $\mathcal{F}_1(x)$ defined in   \eqref{refined ramanujan function}. For $\mathrm{Re}(x)>0$, we have

\begin{enumerate}
\item 
\begin{align}\label{connection eqn2}
\mathcal{J}_1(x)=3\mathcal{F}_1(x)-2\mathcal{F}_1(2x)-\mathcal{F}_1\left(\frac{x}{2}\right) +\frac{1+x}{2x}\log(2).
\end{align}  

\item  
\begin{align}\label{J 2term}
\mathcal{J}_1(x)-\frac{1}{x}\mathcal{J}_1\left(\frac{1}{x}\right)=0.
\end{align}
\end{enumerate}
\end{theorem}


\medskip

We next show that $\mathcal{J}_1(x)$ at natural arguments can be given explicitly in terms of logarithms.
\begin{proposition}\label{J at naturals}
Let $n$ be any natural number. Then, we have
\begin{align}\label{J at naturals n}
\mathcal{J}_1(n)=\frac{1 }{n}\log(2)-\frac{1}{n}\sum_{j=1}^n{\vphantom{\sum}}'\frac{\log\left(\frac{1}{2}-\frac{1}{2}\exp\left(\frac{\pi i(2j-1)}{n}\right)\right)}{1+\exp\left(\frac{\pi i(2j-1)}{n}\right)},
\end{align}
and 
\begin{align}\label{J at naturals 1/n}
\mathcal{J}_1\left(\frac{1}{n}\right)= \log{2}-\sum_{j=1}^n{\vphantom{\sum}}'\frac{\log\left(\frac{1}{2}-\frac{1}{2}\exp\left(\frac{\pi i(2j-1)}{n}\right)\right)}{1+\exp\left(\frac{\pi i(2j-1)}{n}\right)},
\end{align}
where the primed summation means that if $n$ is odd then we exclude the term $j=(n+1)/2$ and add $1/2$ to the total sum. In the case when $n$ is even then it is a usual sum without any restrictions.
\end{proposition}

As an application of our functional equations and Kronecker limit formula, we provide the following result.
\begin{proposition}\label{J evaluation}
Let $n\geq1$ be an even number and $u=n+\sqrt{n^2-1}$. Then $\mathcal{J}_1(u)$ can be given by
\begin{align*}
\mathcal{J}_1(u)&=\frac{1}{u-1} C(u)-\frac{2\gamma}{u+1}+\frac{u^2+1}{u(u+1)}\log(2)-\frac{2}{u+1}\log\left(\frac{u-1}{u+1}\right),
\end{align*}
where $C(u)$ is given in terms of the Kronecker limit formula:
\begin{align*}
C(u)=\widetilde{P}\left(u,\frac{1}{u}\right)-2\widetilde{P}\left(\frac{u+1}{2},\frac{u+1}{2u}\right)-2\widetilde{P}\left(\frac{2u}{u+1},\frac{2}{u+1}\right),
\end{align*}
with $\widetilde{P}(x,y)$ being given in \eqref{p1 in terms of f}.
\end{proposition}
 
\section{  Eisenstein series } \label{connectionEisen} 

 
 \subsection{Eisenstein series of  weight $1$  }
Consider  the following Eisenstein series of weight $1$:
\begin{eqnarray*}\label{es1}
    E_1(\tau)=1 - {4}\sum_{m,n\geq 1}  q^{mn},\quad q=e^{2\pi i \tau}, \tau\in \mathbb{H},
\end{eqnarray*}
where $\mathbb{H}:=\{\tau\in\mathbb{C}:\mathrm{Im}(\tau)>0\}$ is the complex upper half plane.

We now provide a connection between the Ramanujan period function $\mathcal{F}_1(\tau)$ and the Eisenstein series   $E_1(\tau)$:
\begin{proposition}\label{Fourier}
For $\tau\in\mathbb{H}$, we have
\begin{eqnarray} \label{ei}
   \mathcal{F}_1(-\tau)-\mathcal{F}_1( \tau)
   = \frac{\pi i}{2}  {E_1}(\tau). 
\end{eqnarray}
\end{proposition}

\medskip

\subsection{Cohomological aspects}

Now Proposition \ref{Fourier} and the functional equation of $\mathcal{F}_1(\tau)$ in (\ref{cleaner functional equations}) imply that,
 \begin{eqnarray*}
   \mathcal{F}_1(-\tau) - \mathcal{F}_1( \tau)=  -\mathcal{F}_1(\tau)-\frac{1}{\tau}\mathcal{F}_1\left(-\frac{1}{\tau}\right)=\frac{\pi i}{2} E_1(\tau),
\end{eqnarray*}
so that
\begin{eqnarray}\label{e1f1}
 \frac{1}{\tau}E_1\left(-\frac{1}{\tau}\right) -  E_1(\tau) 
     =\frac{4}{ \pi i}  \mathcal{F}_1(\tau), 
\end{eqnarray}
Since $E_1(\tau)$ is periodic so $T\rightarrow 0$ and $S\rightarrow \mathcal{F}_1(\tau)$ define  a $1$-cocycle with values in the space of analytic functions on $\mathbb{H}$.  Now using \eqref{e1f1} and the periodicity of $E_1(\tau)$, we can see that 
\begin{eqnarray}\label{another}
\mathcal{F}_1(\tau)=\mathcal{F}_1(\tau+1)+\frac{1}{\tau+1}\mathcal{F}_1\left(\frac{\tau}{\tau+1}\right),\quad \tau\in\mathbb{H}.
\end{eqnarray}
After analytic continuation of $\mathcal{F}_1(\tau)$  from $\mathbb{H}$ to $\mathbb{C} \backslash (-\infty, 0] $,   the result in (\ref{another}) gives
  another proof of period relation of $\mathcal{F}_1(\tau)$ proved in Theorem \ref{cleaner functional equations}.
 
\bigskip

The relation in (\ref{e1f1})   can be considered more generally by considering $\mathcal{F}_s(x), s\in \mathbb{C}.$ Corresponding results and other properties will be investigated in Section \ref{generalized RPF}. 

\medskip
 
\subsection{Special value of $\mathcal{F}_1(x)$  at Rational points }
Kurukawa \cite{kuru} studied the following function
\begin{align}\label{kuruR1}
R_1(\tau)=-\frac{1}{4}\left(E_1\left(-\frac{1}{\tau}\right)-\tau E_1(\tau)\right),\qquad \tau\in\mathbb{H},
\end{align}
where $E_1(\tau)$ is the weight $1$ Eisenstein series which is defined in \eqref{e1f1}.

Equations \eqref{e1f1} and \eqref{kuruR1} together implies that
\begin{align}
R_1(\tau)= -\frac{\tau}{\pi i}\mathcal{F}_1(\tau).
\end{align}

Therefore, we can say that the function $\mathcal{F}_1(\tau)$ is analytic continuation of the Kurukawa function $R_1(\tau)$ in the region $\tau\in\mathbb{C}\backslash(-\infty,0].$

Moreover, Kurukawa \cite[Theorem 1]{kuru} evaluated the function $R_1(\tau)$ at cusps $\tau\to N$ and $\tau\to 1/N$, where $N$ is a natural number. Hence, his theorem provides special values of the Ramanujan period function at natural numbers $\tau=N$ and $\tau=1/ N$:
\begin{theorem}
Let $\mathcal{F}_1(x)$ be defined in \eqref{refined ramanujan function}.   For any natural number $N$ we have 
\begin{align}\label{n}
\mathcal{F}_1(N)=\frac{\pi}{2N}\left\{\frac{1}{\pi}-\frac{1}{N}\sum_{k=1}^{\lfloor N/2\rfloor}(N-2k)\cot\left(\frac{\pi k}{N}\right)\right\},
\end{align}
 and 
\begin{align}\label{1/n}
\mathcal{F}_1\left(\frac{1}{N}\right)=\frac{\pi}{2}\left\{\frac{1}{\pi}-\frac{1}{N}\sum_{k=1}^{\lfloor N/2\rfloor}(N-2k)\cot\left(\frac{\pi k}{N}\right)\right\}.
\end{align}
\end{theorem}
Note that \eqref{1/n} is just a simple consequence of \eqref{n} and the two-term functional equation of $\mathcal{F}_1(\tau)$ given in \eqref{refined 2-term}.

\medskip


\section{As a limit of a  Period function  of Lewis-Zagier }\label{generalized RPF}

This section recalls a period function  introduced by Lewis and Zagier \cite{lz}. Here we present several new properties of their function and  show that its limiting value at $s=\frac{1}{2}$ is the Ramanujan period function $ \mathcal{F}_1(x)$.

Lewis and Zagier \cite[p.~228, Example 2]{lz} defined the following  function:
\begin{align}\label{psi+}
\psi^+_s(\tau):={\sum_{m,n\geq0}}^*\frac{1}{(m\tau+n)^{2s}},\qquad (\mathrm{Re}(s)>1,\ \tau\in\mathbb{H}),
\end{align}
where  $*$ on the summation sign means that the term  $m = n = 0$ is to be omitted and the terms with either $m$ or $n$ equal to $0$ are to be counted with multiplicity $\frac{1}{2}$.

The next proposition provides some properties of the function $\psi^+_s(\tau)$.
\begin{proposition}\label{properties of GRPF}
Let $x\in\mathbb{C}\backslash(-\infty,0]$ and  let $\zeta(s,x)=\sum_{m=0}^{\infty}\frac{1}{(m+x)^s}, \mathrm{Re}(s)>1$, be the Hurwitz zeta function. 
\begin{enumerate}
\item  For  \textup{Re}$(s)>1$, we have\footnote{The function on the right-hand side of \eqref{f-psi} has been studied by Dixit \cite[Theorem 1.5]{dixitijnt}, where, among other things, he obtained a two-term functional equation for it.}
\begin{align}\label{f-psi}
\psi^+_s(x)=\sum_{n=1}^\infty\left(\zeta(2s,nx)-\frac{1}{2}(nx)^{-2s}+\frac{1}{1-2s}(nx)^{1-2s}\right)+\frac{1}{2}\zeta(2s)+\frac{x^{1-2s}}{2s-1}\zeta(2s-1).
\end{align}
\item Let $\mathrm{Re}(x)>0$ and $\mathrm{Re}(s)>1$. Then
\begin{align}\label{integral representation for GRPF}
\psi^+_s(x)=\frac{1}{\Gamma(2s)}\int_0^\infty\left(\frac{1}{e^t-1}+\frac{1}{2}+\frac{x^{2s}}{2}\right)\frac{t^{2s-1}}{e^{xt}-1}dt.
\end{align}
\item Let $x\in\mathbb{C}\backslash(-\infty,0]$. The function $\psi^+_s(x)$  can be extended analytically as a function of $s$ in the whole complex plane except having a simple pole at $s=1$ with residue $\frac{1}{2x}$.  More precisely, it has the following Laurent series expansion around $s=1$:
\begin{align*}
\psi_s^+(x)=\frac{1/(2x)}{s-1}+\left\{\frac{\gamma-\log x-1}{x}+\frac{\pi^2}{12}\left(1-\frac{1}{x^2}\right)+\sum_{n=1}^\infty\left(\zeta(2,nx)-\frac{1}{nx}\right)\right\}+\mathcal{O}(s-1).
\end{align*} 
\item The function $\psi^+_s(x)$ reduces to the Ramanujan period function $\mathcal{F}_1(x) $ in a special case:
\begin{align*}
\lim_{s\to1/2}\psi_s^+(x)=-\mathcal{F}_1(x).
\end{align*}
Hence we refer the function $\psi_s^+(x)$ as the generalized Ramanujan period function.
\end{enumerate}
\end{proposition}

Lewis and Zagier \cite{lz} also showed that the function  $\psi^+_s(x)$ satisfies the following three-term relation:
\begin{align*}
\psi^+_s(x)-\psi^+_s(x+1)-\frac{1}{(x+1)^{2s}}\psi^+_s\left(\frac{x}{x+1}\right)=0.
\end{align*}

Following are the growth conditions for $\psi^+_{s}(x)$, given by Lewis and Zagier \cite[p.~244]{lz}:

As $x\to0$
\begin{align}
{\psi}^+_{s}(x)\sim\frac{x}{12}\Gamma(2s+1)+\frac{1}{2x^{2s}}\zeta(2s)+\frac{1}{x(2s-1)}\zeta(2s-1),\nonumber
\end{align}
and as $x\to\infty$,
\begin{align}
{\psi}^+_{s}(x)\sim\frac{1}{2}\zeta(2s)+\frac{x^{1-2s}}{2s-1}\zeta(2s-1)+\frac{1}{12x^{2s+1}}\Gamma(2s+1).\nonumber
\end{align}

\subsection{Hecke Eigenform}
In this section we show that the  period function $\psi_s^+(x)$ is a Hecke eigenform under the action of $\widetilde{T}_n$.

\begin{theorem}\label{hecke operator functional equationss}
Let  $\widetilde{T}_n=\sum_{\gamma}v_\gamma\gamma$  be the $n$-th Hecke operator, as it defined in (\ref{existence}), acting  on the space of period functions and that all matrices in $\widetilde{T}_n$ have non-negative entries. Let $x>0$. Then the   period function $\psi_s^+(x)$ is a Hecke eigenform of the Hecke operator $\widetilde{T}_n$\textup{:}
\begin{align*}
(\psi_s^+|_{2s}\widetilde{T}_n)(x)=n^{s}\sigma_{1-2s}(n)\psi_s^+(x),
\end{align*}
where $\sigma_s(n)$ is the generalized divisor function $\sigma_s(n):=\sum_{d|n}d^s$.
\end{theorem}

Note that, in the special case $s\to1/2$, the above theorem reduces to the result for the Ramanujan period function $\mathcal{F}_1(x)$ as stated in Theorem \ref{hecke operator functional equations}.

\bigskip

\subsection{Connection to the work of Bettin and Conrey}

For $s\in \mathbb{C},$  consider Eisenstein series  
\begin{align}\label{eisen}
E_{s}(x)=1+\frac{2}{\zeta(1-s)}  \sum_{n=1}^\infty\sigma_{s-1}(n)q^n, q=e^{2\pi i x},  x \in \mathbb{H}. 
\end{align}
Bettin and Conrey \cite{bc} defined the period function:
\begin{align*}
\Psi_s(x):=E_{2s}(x)-\frac{1}{x^{2s}}E_{2s}\left(-\frac{1}{x}\right),\qquad\tau\in\mathbb{H},
\end{align*}
and showed that it satisfies the two term and three term period relations:
\begin{align*}
\Psi_s(x)-\Psi_s(x+1)-\frac{1}{(x+1)^{2s}}\Psi_s\left(\frac{x}{x+1}\right)=0.
\end{align*}

We conclude this subsection by demonstrating that the function $\Psi_s(\tau)$ is, in fact, equivalent to the   period function $\psi_s^+(\tau):$ 
\begin{proposition}\label{RaBC}
Let $\mathrm{Re}(s)>1$ and $\tau\in\mathbb{H}$, we have
\begin{align*}
\Psi_{s}(\tau)=\frac{ 2 (1- e^{ -2\pi i s})}{(1+e^{-2\pi i s}) \zeta(2s)}\psi_s^+(\tau).
\end{align*}
\end{proposition}

%


\section{Proofs}

We first provide an integral representation for the function $\mathscr{F}_1(x)$ which is imperative to prove some of our results. 
\begin{lemma}\label{int} 
Let the function $\mathscr{F}_1(x)$ be defined in \eqref{ramanujan function}. Let $\mathrm{Re}(x)>0$, then
\begin{eqnarray}\label{integral representation for Ramanujan}
 \mathscr{F}_1(x)=-\int_0^\infty\left(\frac{1}{e^{t}-1}-\frac{1}{t}+\frac{1}{2}\right)\frac{1}{e^{xt}-1}dt.
\end{eqnarray}
\end{lemma}
\begin{proof}
Using the well-known result \cite[p.~248]{ww}
\begin{align*}
\psi(x)-\log(x)+\frac{1}{2x}=-\int_0^\infty\left(\frac{1}{e^{t}-1}-\frac{1}{t}+\frac{1}{2}\right)e^{-xt}dt, 
\end{align*}
with replacing $x$ by $nx$ and then summing over $n$, we obtain 
\begin{align*}
\mathscr{F}_1(x)&=-\sum_{n\geq1}\int_0^\infty\left(\frac{1}{e^{t}-1}-\frac{1}{t}+\frac{1}{2}\right)e^{-nxt}dt\\
&=-\int_0^\infty\left(\frac{1}{e^{t}-1}-\frac{1}{t}+\frac{1}{2}\right)\frac{1}{e^{xt}-1}dt,
\end{align*}
where the last step follows upon interchanging the order of summation and integration, which is justified due the absolute convergence. This proves our \eqref{integral representation for Ramanujan}.
\end{proof}

\medskip
\subsection{Period relations}
\begin{proof}[\textbf{Theorem \textup{\ref{cleaner functional equations}}}][] 
After rephrasing the Ramanujan's result \eqref{ramanujan fe}, we get
\begin{align*}
    \mathscr{F}_1(x)-\frac{1}{x}\mathscr{F}_1\left(\frac{1}{x}\right)=\frac{1}{2}\left(\gamma-\log\left(\frac{2\pi}{x}\right)\right)-\frac{1}{2x}\left(\gamma-\log\left(2\pi x\right)\right).
\end{align*}
It is now easy to obtain \eqref{refined 2-term} using the fact $\mathscr{F}_1(x)=\mathcal{F}_1(x)+\frac{1}{2}\left(\gamma-\log\left(\frac{2\pi}{x}\right)\right)$ in the above equation and simplifying the expressions.

We next provide a proof of the three-term functional equation \eqref{refined 3-term}. We first assume $x>1$, and then we extend the proof through analytic continuation. 

Employing the fact that
\begin{align}
&\frac{1}{\left(e^{t}-1\right)\left(e^{xt}-1\right)}
 =\frac{1}{\left(e^{t}-1\right)\left(e^{(x-1)t}-1\right)}-\frac{1}{\left(e^{xt}-1\right)\left(e^{(x-1)t}-1\right)}-\frac{1}{e^{xt}-1}\nonumber
\end{align}
in the integral representation of $\mathscr{F}_1(x)$ given in  (\ref{integral representation for Ramanujan}), we arrive at
\begin{align}
 \mathscr{F}_1(x)
& =-\int_0^\infty
 \left\{ 
\frac{1}{\left(e^{t}-1\right)\left(e^{(x-1)t}-1\right)}-\frac{1}{\left(e^{xt}-1\right)
\left(e^{(x-1)t}-1\right)}
-\frac{1}{e^{xt}-1} -\frac{1}{t\left(e^{xt}-1\right)}\right.\nonumber\\
 &\qquad\qquad\left.+\frac{1}{2\left(e^{xt}-1\right)}
 \right\}  dt.\nonumber
\end{align}
We next rewrite the integrand of the above expression as 
\begin{align}\label{withax}
\mathscr{F}_1(x)&=-\int_0^\infty\left\{\left(\frac{1}{e^{t}-1}-\frac{1}{t}+\frac{1}{2}\right)\frac{1}{e^{(x-1)t}-1}-\left(\frac{1}{e^{xt}-1}-\frac{1}{xt}+\frac{1}{2}\right)\frac{1}{e^{(x-1)t}-1}\right.\nonumber\\
&\qquad\left.+\left(\frac{1}{t\left(e^{(x-1)t}-1\right)}-\frac{1}{xt\left(e^{(x-1)t}-1\right)}-\frac{1}{t\left(e^{xt}-1\right)}-\frac{1}{2\left(e^{xt}-1\right)}\right)\right\}dt\nonumber\\
&=\mathscr{F}_1(x-1)-\frac{1}{x}\mathscr{F}_1\left(\frac{x-1}{x}\right)-B(x),
\end{align}
where the last step follows invoking the definition of $\mathscr{F}_1(x)$. Here
\begin{align*}
B(x):=\int_0^\infty\left\{\frac{(x-1)/x}{t\left(e^{(x-1)t}-1\right)}-\left(\frac{1}{t}+\frac{1}{2}\right)\frac{1}{\left(e^{xt}-1\right)}\right\}dt.
\end{align*}
It remains to evaluate the integral $B(x)$. To that end, consider the integral
\begin{align*}
B(x;s) :=\int_0^\infty\left\{\frac{(x-1)/x}{t\left(e^{(x-1)t}-1\right)}-\left(\frac{1}{t}+\frac{1}{2}\right)\frac{1}{\left(e^{xt}-1\right)}\right\}t^sdt,
\end{align*}
for Re$(s)\geq0$. To separate the integrals, we first assume Re$(s)>1$. So that, we have 
\begin{align}\label{axsa}
B(x;s)&=\frac{x-1}{x}\int_0^\infty\frac{t^{s-1}}{e^{(x-1)t}-1}dt-\int_0^\infty\frac{t^{s-1}}{e^{xt}-1}dt-\frac{1}{2}\int_0^\infty\frac{t^s}{e^{xt}-1}dt\nonumber\\
&=\frac{1}{x(x-1)^{s-1}}\int_0^\infty\frac{t^{s-1}}{e^{t}-1}dt-\frac{1}{x^{s}}\int_0^\infty\frac{t^{s-1}}{e^{t}-1}dt-\frac{1}{2x^{s+1}}\int_0^\infty\frac{t^s}{e^{t}-1}dt,
\end{align}
where the last step follows by making the change of variable $t\rightarrow t/(x-1)$ in the first integral, and $t\rightarrow t/x$ in each of the last two integrals. By the help of the well-known formula \cite[p.~266]{ww}
\begin{align}\label{gammazeta}
\Gamma(z)\zeta(z)=\int_0^\infty\frac{t^{z-1}}{e^t-1}dt\quad(\mathrm{Re}(z)>1),
\end{align}
in \eqref{axsa}, we see that
\begin{align}\label{befreos0}
B(x;s)&=\frac{1}{x(x-1)^{s-1}}\Gamma(s)\zeta(s)-\frac{\Gamma(s)\zeta(s)}{x^s}-\frac{1}{2x^{s+1}}\Gamma(s+1)\zeta(s+1).
\end{align}
Note that the right-hand side of the above equation is analytic for Re$(s)\geq0$ and has a removable singularity at $s=0$. Therefore, we want to take $s\to0$ on both sides of \eqref{befreos0}. To that end, we expand all the functions as their Laurent series expansions around $s=0$:
\begin{align*}
\Gamma(s)&=\frac{1}{s}-\gamma+\mathcal{O}(s),&
\Gamma(s+1)&=1-\gamma s+\mathcal{O}(s^2),\\
\zeta(s)&=-\frac{1}{2}-\frac{s}{2}\log(2\pi)+\mathcal{O}(s^2),&
\zeta(s+1)&=\frac{1}{s}+\gamma+\mathcal{O}(s),\\
(x-1)^{-s}&=1-s\log(x-1)+\mathcal{O}(s^2),&
x^{-s}&=1-s\log(x)+\mathcal{O}(s^2).
\end{align*}
Therefore,
\begin{align*}
B(x,0)&=\lim_{s\to0}\left\{\frac{x-1}{x}\left(\frac{1}{s}-\gamma+\mathcal{O}(s)\right)\left(1-s\log(x-1)+\mathcal{O}(s^2)\right)\left(-\frac{1}{2}-\frac{s}{2}\log(2\pi)+\mathcal{O}(s^2)\right)\right.\nonumber\\
&\qquad\left.\quad-\left(1-s\log(x-1)+\mathcal{O}(s^2)\right)\left(\frac{1}{s}-\gamma+\mathcal{O}(s)\right)\left(-\frac{1}{2}-\frac{s}{2}\log(2\pi)+\mathcal{O}(s^2)\right)\right.\nonumber\\
&\qquad\left.\quad-\frac{1}{2x}\left(1-s\log(x)+\mathcal{O}(s^2)\right)\left(1-\gamma s+\mathcal{O}(s^2)\right)\left(\frac{1}{s}+\gamma+\mathcal{O}(s)\right)\right\}.
\end{align*}
We now observe that the singular terms containing $1/s$ cancel each other. Hence after simplifying  the expressions, we arrive at
\begin{align}\label{ax0}
B(x;0)=-\frac{1}{2x}\left\{\gamma-(x-1)\log(x-1)+x\log(x)-\log(2\pi x)\right\}.
\end{align}
Substituting the value of $B(x;0)=B(x)$ from \eqref{ax0} in \eqref{withax}, we see that 
\begin{align*}
&\mathscr{F}_1(x)-\mathscr{F}_1(x-1)+\frac{1}{x}\mathscr{F}_1\left(\frac{x-1}{x}\right)\\
&=\frac{1}{2x}\left\{\gamma-(x-1)\log(x-1)+x\log(x)-\log(2\pi x)\right\}.
\end{align*}
Replacing $x$ by $x+1$ in the above equation and simplifying, we obtain
\begin{align}\label{threeterm}
&\mathscr{F}_1(x)-\mathscr{F}_1(x+1)-\frac{1}{x+1}\mathscr{F}_1\left(\frac{x}{x+1}\right)=-\frac{1}{2(x+1)}\left\{\gamma-\log(2\pi)-x\log\left(\frac{x}{x+1}\right)\right\}.
\end{align}
This equation holds for $x>0$. Since both sides of \eqref{threeterm} are analytic in the region $x \in \mathbb{C} \setminus (-\infty,0]$, we can conclude, by analytic continuation, that \eqref{threeterm} holds for $x\in \mathbb{C} \setminus (-\infty,0]$. Now again employing $\mathscr{F}_1(x)=\mathcal{F}_1(x)+\frac{1}{2}\left(\gamma-\log\left(\frac{2\pi}{x}\right)\right)$ in \eqref{threeterm} and simplifying the expressions, we complete the proof of \eqref{refined 3-term}.
\end{proof}

\medskip
\begin{proof}[\textbf{Theorem \textup{\ref{asymptotics}}}][] 
For $|\arg(x)|<\pi$, as $x\to\infty$, we have \cite[p.~259, formula~6.3.18]{as},
\begin{align*}
\psi(x)\sim\log(x)+\sum_{p=1}^\infty\zeta(1-p)x^{-p}.
\end{align*}
This gives
\begin{align*}
\psi(x)-\log(x)+\frac{1}{2x}\sim\sum_{p=2}^\infty\zeta(1-p)x^{-p}.
\end{align*}
Replacing $x$ by $nx$ in the above formula and using the definition of $\mathcal{F}_1(x)$ from \eqref{refined ramanujan function}, we obtain \eqref{xtendinfinity}.

Equation \eqref{xtendzero} follows easily after invoking functional equation \eqref{refined 2-term} and employing \eqref{xtendinfinity} with replacing $x$ by $1/x$.
\end{proof}

\subsection{ Hecke Eigenform}
Let $\mathcal{M}_n$ be defined in Section \ref{hecke operator section}.
 Now define $\mathcal{M}=\cup_{n\in \mathbb{N}}\mathcal{M}_n$ and let $\mathcal{R}=\mathbb{Q}[\mathcal{M}]=\oplus_{n\in \mathbb{N}}\mathcal{R}_n,$ which is a graded ring.   This ring acts on the right on the vector space of two variable functions on $\mathbb{C}^2$ by the formula
\begin{align}\label{facts}
\left(f\circ \sum_{i}\lambda_i \begin{pmatrix}
a & b \\
c & d 
\end{pmatrix}\right)(x,y)=\sum_i \lambda_if(a_ix+b_iy, c_ix+d_iy).
\end{align}

We need the following result of Radchenko and Zagier \cite[Theorem 1]{raza}:
\begin{proposition}\label{razaprop}
Define
\begin{align}\label{C defn}
\mathscr{C}(x,y):=C(x)C(y)+1,
\end{align}
where $C(x)$ is defined as
\begin{align*}
C(x):=
\begin{cases}
\cot(\pi x),& x\in\mathbb{C}\backslash\mathbb{Z},\\
0, & x\in\mathbb{Z}.
\end{cases}
\end{align*}
Take the $n$th Hecke operator  $\widetilde{T}_n\in\mathcal{R}_n$. Then
\begin{eqnarray*} 
(\mathscr{C}\circ \widetilde{T}_n)(x, y)-\sum_{\ell |n} \ell \mathscr{C}(\ell x, \ell y)=c(\widetilde{T}_n),
\end{eqnarray*}
where $c:\mathcal{R}\rightarrow \mathbb{Z}$ is the group homomorphism defined on generators by
\begin{eqnarray*}
\begin{bmatrix}
a & b \\
c & d 
\end{bmatrix} \rightarrow 
\begin{cases} 
    1- sgn(a+b) sgn(c+d), & a+b\neq 0, c+d\neq 0, \\
1- sgn(a+b) sgn( d), & a+b\neq 0, c+d= 0, \\
1- sgn( b) sgn(c+d), & a+b= 0, c+d\neq 0.
   \end{cases}
\end{eqnarray*}
\end{proposition}

\begin{proof}[\textbf{Theorem \textup{\ref{hecke operator functional equations}}}][]
For Re$(s)>2$ and $x>0$, let us define 
\begin{align}
A(x,s):=\int_0^\infty\frac{t^{s-1}}{\left(e^{xt}-1\right)\left(e^{t}-1\right)}dt.\nonumber
\end{align}
Note that we can rewrite the integral above as
\begin{align}
A(x,s)&=x^{1-s}\Gamma(s-1)\zeta(s-1)-\frac{1}{2}x^{-s}\Gamma(s)\zeta(s)+\int_0^\infty\left(\frac{1}{e^{t}-1}-\frac{1}{t}+\frac{1}{2}\right)\frac{t^{s-1}}{e^{xt}-1}dt.\nonumber
\end{align}
Observe that the right-hand side is analytic for Re$(s)>0$ except at $s=1$ and $s=2$ where it has simple poles. Therefore, it gives analytic continuation of the function $A(x,s)$ in the region Re$(s)>0$. It is seen that, as $s\to1$, we have
\begin{align}\label{final A}
&A(x,s)\nonumber\\
&=-\frac{x+1}{2x}\frac{1}{s-1}+\frac{1}{2}\left(\gamma-\log(2\pi/x)+\frac{\log(x)}{x}\right)+\int_0^\infty\left(\frac{1}{e^{t}-1}-\frac{1}{t}+\frac{1}{2}\right)\frac{1}{e^{xt}-1}dt+\mathcal{O}(s-1)\nonumber\\
&=-\frac{x+1}{2x}\frac{1}{s-1}+\frac{1}{2}\left(\gamma-\log(2\pi/x)+\frac{\log(x)}{x}\right)-\mathscr{F}_1(x)+\mathcal{O}(s-1),
\end{align}
where in the last we used \eqref{integral representation for Ramanujan}.
Equations \eqref{refined ramanujan function} and \eqref{final A} together yield
\begin{align}\label{axs1}
A(x,s)&=-\frac{x+1}{2x}\frac{1}{s-1}-\mathcal{F}_1(x)+\frac{\log(x)}{2x}+\mathcal{O}(s-1).
\end{align}

We next aim to express the integral
\begin{align*}
\int_0^\infty\mathscr{C}(ixt,iyt)t^{s-1}dt
\end{align*}
in terms of the function $A(x,s)$. To that end, using the definition of $\mathscr{C}(x,y)$ given in \eqref{C defn}, we see that
\begin{align*}
\int_0^\infty\mathscr{C}(ixt,iyt)t^{s-1}dt&=\int_0^\infty\left(1+\cot(\pi ixt)\cot(\pi iyt)\right)t^{s-1}dt\\
&=\frac{1}{y^s}\int_0^\infty\left(1+\cot(\pi ixt/y)\cot(\pi it)\right)t^{s-1}dt,
\end{align*}
where we employed the change of variable $t\rightarrow t/y$. Now using the expression $\cot x=i\frac{e^{ix}+e^{-ix}}{e^{ix}-e^{-ix}}$ in the above equation and simplifying, we are led to
\begin{align*}
\int_0^\infty\mathscr{C}(ixt,iyt)t^{s-1}dt&=-\frac{2}{y^s}\int_0^\infty\frac{e^{-\pi xt/y}e^{\pi t}+e^{\pi xt/y}e^{-\pi t}}{\left(e^{-\pi xt/y}-e^{\pi xt/y}\right)\left(e^{-\pi t}-e^{\pi t}\right)}t^{s-1}dt.
\end{align*} 
By adding and subtracting the term $e^{-\pi xt/y}e^{-\pi t}$ in the numerator of the integrand on the right-hand side, we obtain
\begin{align*}
&\int_0^\infty\mathscr{C}(ixt,iyt)t^{s-1}dt\\
&=-\frac{2}{y^s}\int_0^\infty\frac{e^{-\pi xt/y}\left(e^{\pi t}-e^{-\pi t}\right)+2e^{-\pi xt/y}e^{-\pi t}+e^{-\pi t}\left(e^{\pi xt/y}-e^{-\pi xt/y}\right)}{\left(e^{-\pi xt/y}-e^{\pi xt/y}\right)\left(e^{-\pi t}-e^{\pi t}\right)}t^{s-1}dt\\
&=-\frac{2}{y^s}\int_0^\infty\frac{t^{s-1}dt}{e^{2\pi xt/y}-1}-\frac{4}{y^s}\int_0^\infty\frac{t^{s-1}dt}{\left(e^{2\pi xt/y}-1\right)\left(e^{2\pi t}-1\right)}-\frac{2}{y^s}\int_0^\infty\frac{t^{s-1}dt}{e^{2\pi t}-1}\\
&=-2\frac{(2\pi)^{-s}}{x^{s}}\Gamma(s)\zeta(s)-\frac{4}{y^s}(2\pi)^{-s}A\left(\frac{x}{y},s\right)-\frac{2}{y^s}(2\pi)^{-s}\Gamma(s)\zeta(s),
\end{align*}
where in the last step we invoked \eqref{gammazeta}. This now implies that
\begin{align}\label{integral1}
&-\frac{(2\pi)^{s}}{4}\int_0^\infty\mathscr{C}(ixt,iyt)t^{s-1}dt=\frac{1}{y^s}A\left(\frac{x}{y},s\right)+\frac{1}{2}\Gamma(s)\zeta(s)\left(\frac{1}{x^s}+\frac{1}{y^s}\right).
\end{align}
We wish to take $s\to1$ in \eqref{integral1}. Thus, using the fact that 
\begin{align*}
&\frac{1}{2}\Gamma(s)\zeta(s)\left(\frac{1}{x^s}+\frac{1}{y^s}\right)=\frac{x+y}{2xy}\frac{1}{s-1}-\frac{\log(x)}{2x}-\frac{\log(y)}{2y}+\mathcal{O}(s-1), \quad s\to1,
\end{align*}
and invoking \eqref{axs1},   we are led to
\begin{align}\label{eval in terms of F1231}
-\frac{(2\pi)^{s}}{4}\int_0^\infty\mathscr{C}(ixt,iyt)t^{s-1}dt=-\frac{1}{y}\mathcal{F}_1\left(\frac{x}{y}\right)+\mathcal{O}(s-1).
\end{align}
We apply $\widetilde{T}_n$ on both sides of the above equation. We first evaluate the left-hand side.

Employing Proposition \ref{razaprop} and noting that $c(\widetilde{T}_n)=0$ (as we assumed that all the entries are non-negative), we see that
\begin{align}\label{lhs}
-\frac{(2\pi)^{s}}{4}\int_0^\infty\mathscr{C}(ixt,iyt)\circ\widetilde{T}_n\ t^{s-1}dt
&=-\frac{(2\pi)^{s}}{4}\int_0^\infty t^{s-1}\sum_{\ell |n}\ell\mathscr{C}(i\ell  xt,i\ell yt) dt\nonumber\\
&=-\frac{(2\pi)^{s}}{4}\sum_{\ell |n}\ell\int_0^\infty t^{s-1}\mathscr{C}(i\ell  xt,i\ell yt) dt\nonumber\\
&=\sum_{\ell |n}\ell\left\{-\frac{1}{\ell y}\mathcal{F}_1\left(\frac{x}{y}\right)\right\}+\mathcal{O}(s-1)\nonumber\\
&=-d(n)\frac{1}{y}\mathcal{F}_1\left(\frac{x}{y}\right)+\mathcal{O}(s-1),
\end{align}
where the penultimate step we used \eqref{eval in terms of F1231}.

Next we simplify the right-hand side of \eqref{eval in terms of F1231} after applying $\widetilde{T}_n$. 
\begin{align}\label{rhs}
\left(\frac{1}{y}\mathcal{F}_1\left(\frac{x}{y}\right)\right)\circ\widetilde{T}_n&=\left(\frac{1}{y}\mathcal{F}_1\left(\frac{x}{y}\right)\right)\circ\sum_{\gamma}v_\gamma \gamma\nonumber\\
&=\sum_{\gamma}v_\gamma\left(\frac{1}{y}\mathcal{F}_1\left(\frac{x}{y}\right)\right)\circ\gamma\nonumber\\
&=\sum_{\gamma}v_\gamma\left\{\frac{1}{cx+dy}\mathcal{F}_1\left(\frac{ax+by}{cx+dy}\right)\right\},
\end{align}
where we used \eqref{facts}. Therefore Equations \eqref{eval in terms of F1231}, \eqref{lhs} and \eqref{rhs}, after taking $s\to1$ and $y=1$, imply
\begin{align}
\sum_{\gamma}v_\gamma\left\{\frac{1}{cx+d}\mathcal{F}_1\left(\frac{ax+b}{cx+d}\right)\right\}=d(n)\mathcal{F}_1\left(x\right).\nonumber
\end{align}
This completes the proof of the theorem after multiplying the both sides by $\sqrt{n}$ and using the definition of the slash operator.
\end{proof}

 \medskip

\subsection{ Kronecker limit formula}

\begin{proof}[\textbf{Theorem \textup{\ref{analyticity>2}}}][]

We employ the Poisson summation formula in the following form \cite[pp.~60-61]{titch1}
\begin{align*}
\frac{1}{2}f(0)+\sum_{q\geq1}f(q)=\int_0^\infty f(t)dt+2\sum_{\ell\geq1}\int_0^\infty f(t)\cos(2\pi\ell t)dt.
\end{align*}
Letting $f(t)=\frac{(x-y)^s}{(px+t)^s(py+t)^s}$, $\mathrm{Re}(s)>1/2$, in the above formula so as to obtain
\begin{align*}
&\frac{1}{2p^{2s}}\frac{(x-y)^s}{(xy)^s}+\sum_{q\geq1}\frac{(x-y)^s}{(px+q)^s(py+q)^s}\nonumber\\
&=\int_0^\infty\frac{(x-y)^s}{(px+t)^s(py+t)^s}dt+2\sum_{\ell\geq1}\int_0^\infty\frac{(x-y)^s\cos(2\pi\ell t)}{(px+t)^s(py+t)^s}dt.
\end{align*}
We add and subtract the term $q=0$ on the left-hand side, and make the change of variable $t\to pt$ in the first integral on the right-hand side and then rephrase the terms so that
\begin{align}\label{iseqn}
&\sum_{q\geq0}\frac{(x-y)^s}{(px+q)^s(py+q)^s}-\frac{1}{2p^{2s}}\frac{(x-y)^s}{(xy)^s}-\frac{1}{p^{2s-1}}I(s)=2\sum_{\ell\geq1}\int_0^\infty\frac{(x-y)^s\cos(2\pi\ell t)}{(px+t)^s(py+t)^s}dt,
\end{align}
where 
\begin{align}\label{I(s)}
I(s):=\int_0^\infty\frac{(x-y)^s}{(x+t)^s(y+t)^s}dt,
\end{align}
which is an analytic function of $s$ in Re$(s)>1/2$.

From \eqref{iseqn}, it is clear that  $$\sum_{q\geq0}\frac{(x-y)^s}{(px+q)^s(py+q)^s}\ll \frac{1}{p^{2s-1}},\quad \mathrm{as}\ p\to\infty.$$
Hence the double series defining the function $\mathcal{Z}(s;x,y)$ in \eqref{Z1} converges absolutely and uniformly on any compact subset of Re$(s)>2$. Therefore, it is an analytic function of $s$ in this region. This proves the first part of the theorem.

We next provide the analytic continuation of the function $\mathcal{Z}(s;x,y)$ in the bigger region. In order to achieve that, we want to know the behaviour of the left-hand side of \eqref{iseqn} as $p\to\infty$. To that end,  we find a bound for the term on the right-hand side of \eqref{iseqn} as $p\to\infty$. For the convenience, we denote this term by $R(p,s)$:
\begin{align*}
R(p,s):&=2\sum_{\ell\geq1}\int_0^\infty\frac{(x-y)^s\cos(2\pi\ell t)}{(px+t)^s(py+t)^s}dt.
\end{align*}
Making the change of variable $t\to pt$ in the first below and then integrating by parts in the subsequent steps, we see that
\begin{align*}
R(p,s)&=2(x-y)^s\frac{1}{p^{2s-1}}\sum_{\ell\geq1}\int_0^\infty\frac{\cos(2\pi\ell pt)}{(x+t)^s(y+t)^s}dt\nonumber\\
&=2(x-y)^s\frac{1}{p^{2s-1}}\sum_{\ell\geq1}\left\{\frac{\sin(2\pi\ell pt)}{2\pi\ell p(x+t)^s(y+t)^s}\Bigg|_0^\infty+\int_0^\infty\frac{\sin(2\pi\ell pt)s(2t+x+y)}{2\pi\ell p(x+t)^{s+1}(y+t)^{s+1}}dt\right\}\nonumber\\
&=(x-y)^s\frac{s}{\pi}\frac{1}{p^{2s}}\sum_{\ell\geq1}\frac{1}{\ell}\int_0^\infty\frac{\sin(2\pi\ell pt)(2t+x+y)}{(x+t)^{s+1}(y+t)^{s+1}}dt\\
&=(x-y)^s\frac{s}{\pi}\frac{1}{p^{2s}}\sum_{\ell\geq1}\frac{1}{\ell} \left\{\frac{-(2t+x+y)\cos(2\pi\ell pt)}{2\pi\ell p(x+t)^{s+1}(y+t)^{s+1}}\Bigg|_0^\infty-\int_0^\infty\frac{\cos(2\pi\ell pt)}{2\pi\ell p(x+t)^{s+2}(y+t)^{s+2}}\right.\nonumber\\
&\left.\qquad\quad\times\left((4s+2)t^2+(s+1)x^2+2sxy+(s+1)y^2+2t(2s+1)(x+y)\right)dt\right\}\\
&=(x-y)^s\frac{s}{2\pi^2}\frac{x+y}{(xy)^{s+1}}\frac{1}{p^{2s+1}}\sum_{\ell\geq1}\frac{1}{\ell^2}-(x-y)^s\frac{s}{2\pi^2}\frac{1}{p^{2s+1}}\sum_{\ell\geq1}\frac{1}{\ell^2}\int_0^\infty\frac{\cos(2\pi\ell pt)}{(x+t)^{s+2}(y+t)^{s+2}}\nonumber\\
&\quad\times\left\{(4s+2)t^2+(s+1)x^2+2sxy+(s+1)y^2+2t(2s+1)(x+y)\right\}dt\\
&\ll (x-y)^{\mathrm{Re}(s)}\frac{|s|\zeta(2)}{2\pi^2}\frac{x+y}{(xy)^{\mathrm{Re}(s)+1}}\frac{1}{p^{2\mathrm{Re}(s)+1}}+(x-y)^{\mathrm{Re}(s)}\frac{|s|\zeta(2)}{2\pi^2}\frac{1}{p^{2\mathrm{Re}(s)+1}}\\
&\quad\times\int_0^\infty\frac{\left|\left\{(4s+2)t^2+(s+1)x^2+2sxy+(s+1)y^2+2t(2s+1)(x+y)\right\}\right|}{(x+t)^{\mathrm{Re}(s)+2}(y+t)^{\mathrm{Re}(s)+2}}dt,
\end{align*}
Note that the integral in the above expression is convergent for Re$(s)>1/2$ and independent of $p$. This implies that, as $p\to\infty$, we have
\begin{align}\label{rbound}
R(p,s)\ll \frac{1}{p^{2\mathrm{Re}(s)+1}}.
\end{align}
Equations \eqref{iseqn} and \eqref{rbound} together yield 
\begin{align}
&\sum_{q\geq0}\frac{(x-y)^s}{(px+q)^s(py+q)^s}-\frac{1}{2p^{2s}}\frac{(x-y)^s}{(xy)^s}-\frac{1}{p^{2s-1}}I(s)\ll \frac{1}{p^{2\mathrm{Re}(s)+1}}.\nonumber
\end{align}
As a result of this fact, we conclude that the series
\begin{align}\label{conv}
\sum_{p\geq1}p^s\left\{\sum_{q\geq0}\frac{(x-y)^s}{(px+q)^s(py+q)^s}-\frac{1}{2p^{2s}}\frac{(x-y)^s}{(xy)^s}-\frac{1}{p^{2s-1}}I(s)\right\},
\end{align}
is absolutely convergent for Re$(s)>0$. 

For Re$(s)>2$, we can rewrite the function $\mathcal{Z}(s;x,y)$ as 
\begin{align}\label{analytic continuation eqn}
\mathcal{Z}(s;x,y)&=\frac{(x-y)^s}{2(xy)^s}\zeta(s)+\zeta(s-1)I(s)+\sum_{p\geq1}p^s\left\{\sum_{q\geq0}\frac{(x-y)^s}{(px+q)^s(py+q)^s}-\frac{1}{p^{2s-1}}I(s)\right.\nonumber\\
&\left.\quad-\frac{1}{2p^{2s}}\frac{(x-y)^s}{(xy)^s}\right\}.
\end{align}
Now, as a consequence of \eqref{conv} and the analyticity of the Riemann zeta function, combined with the fact that $I(s)$, defined in \eqref{I(s)}, is analytic for Re$(s)>1/2$, we conclude that the right-hand side of \eqref{analytic continuation eqn} provides the analytic continuation of $\mathcal{Z}(s;x,y)$ in the half-plane Re$(s)>1/2$. Moreover, it reveals that the function $\mathcal{Z}(s;x,y)$ has only two simple poles at $s=1$ and $s=2$ in this region due to the poles of zeta functions $\zeta(s)$ and $\zeta(s-1)$.

Next we prove \eqref{first kronecker limit formula eqn}. We want to take limit $s\to1$ in \eqref{analytic continuation eqn}. Hence, using the facts
\begin{align*}
&\frac{(x-y)^s}{2(xy)^s}\zeta(s)=\frac{x-y}{2xy}\frac{1}{s-1}+\frac{x-y}{2xy}\left(\gamma+\log\left(\frac{x-y}{xy}\right)\right)+\mathcal{O}(s-1),\ s\to1,
\end{align*}
and 
\begin{align*}
I(1)=\log\left(\frac{x}{y}\right),
\end{align*}
it is clear that
\begin{align}\label{beforepsi}
&\lim_{s\to1}\left(\mathcal{Z}(s;x,y)-\frac{x-y}{2xy}\frac{1}{s-1}\right)\nonumber\\
&=-\frac{x-y}{2xy}\left(\gamma+\log\left(\frac{x-y}{xy}\right)\right)-\frac{1}{2}\log\left(\frac{x}{y}\right)+\sum_{p\geq1}p\left\{\sum_{q\geq0}\frac{(x-y)}{(px+q)(py+q)}-\frac{1}{p}\log\left(\frac{x}{y}\right)\right.\nonumber\\
&\qquad\left.-\frac{1}{2p^{2}}\frac{(x-y)}{xy}\right\}\nonumber\\
&=-\frac{x-y}{2xy}\left(\gamma+\log\left(\frac{x-y}{xy}\right)\right)-\frac{1}{2}\log\left(\frac{x}{y}\right)+\sum_{p\geq1}p\left\{\frac{1}{p}\sum_{q\geq0}\left(\frac{1}{py+q}-\frac{1}{px+q}\right)-\frac{1}{p}\log\left(\frac{px}{py}\right)\right.\nonumber\\
&\qquad\left.-\frac{1}{2p^{2}}\frac{(x-y)}{xy}\right\}\nonumber\\
&=-\frac{x-y}{2xy}\left(\gamma+\log\left(\frac{x-y}{xy}\right)\right)-\frac{1}{2}\log\left(\frac{x}{y}\right)+\sum_{p\geq1}\left\{\left(\psi(px)-\psi(py)\right)-\log\left(\frac{px}{py}\right)-\frac{1}{2py}+\frac{1}{2px}\right\},
\end{align}
where in the last step we used
\begin{align}\label{psix}
\psi(x)=-\gamma+\sum_{q\geq0}\left(\frac{1}{q+1}-\frac{1}{x+q}\right).
\end{align}
Hence equation \eqref{beforepsi} yields
\begin{align}
&\lim_{s\to1}\left(\mathcal{Z}(s;x,y)-\frac{x-y}{2xy}\frac{1}{s-1}\right)\nonumber\\
&=-\frac{x-y}{2xy}\left(\gamma+\log\left(\frac{x-y}{xy}\right)\right)-\frac{1}{2}\log\left(\frac{x}{y}\right)\nonumber\\
&+\sum_{p\geq1}\left\{\left(\psi(px)-\log(px)+\frac{1}{2px}\right)-\left(\psi(py)-\log(py)+\frac{1}{2py}\right)\right\}\nonumber\\
&=-\frac{x-y}{2xy}\left(\gamma+\log\left(\frac{x-y}{xy}\right)\right)+\mathcal{F}_1(x)-\mathcal{F}_1(y),\nonumber
\end{align}
where in the last step we used the definition of the Ramanujan period function $\mathcal{F}_1(x)$ given in \eqref{refined ramanujan function}. This completes the proof of the theorem.
\end{proof}
 
\medskip

\begin{proof}[\textbf{Theorem \textup{\ref{new limit formula}}}][] 
From \cite[pp.~33-34, Lemma 7(iv)]{vz}, for any $T$, one has
$$\mathscr{D}_n\left(\frac{1}{T-x}\right)=\left(\frac{x-y}{(T-x)(T-y)}\right)^{n+1},$$
where the differential operator $\mathscr{D}_n$ is defined in \eqref{doperator}. Employing the fact above, we can calculate that
\begin{eqnarray}
-\mathscr{D}_{k-1}\left(\frac{1}{p^{2k-1}(px+q)}\right)=\left(\frac{(x-y)}{(px+q)(py+q)}\right)^k.\nonumber
\end{eqnarray} 
This implies that 
\begin{eqnarray}\label{derivative of}
-\mathscr{D}_{k-1}\left(\frac{1}{p^{k-1}(px+q)}\right)=\left(\frac{p(x-y)}{(px+q)(py+q)}\right)^k.
\end{eqnarray} 
Note that using \eqref{psix} in the definition of $\mathscr{F}_k(x)$, $k\geq3$, given in \eqref{fk}, we have
\begin{align}
\mathscr{F}_k(x)=\sum_{p\geq1}\frac{1}{p^{k-1}}\left\{-\gamma+\sum_{q\geq0}\left(\frac{1}{q+1}-\frac{1}{px+q}\right)\right\}.\nonumber
\end{align}
We apply the operator $\mathscr{D}_n$ on the above equation and then using \eqref{derivative of}, we arrive at
\begin{align*}
\mathscr{D}_{k-1}(\mathscr{F}_{k})(x,y)=\sum_{p\geq1,q\geq0}\left(\frac{p(x-y)}{(px+q)(py+q)}\right)^k.
\end{align*}
The definition of $\mathcal{Z}(k,x,y)$ given in \eqref{Z1} along with the above fact implies that
\begin{eqnarray}\label{end1}
\mathcal{Z}(k,x,y)=\mathscr{D}_{k-1}(\mathscr{F}_{k})(x,y).
\end{eqnarray}
Hence using \eqref{end1} in the definition of $\widetilde{\zeta}(s,\mathscr{B} )$ given in  \eqref{re} concludes the proof of the theorem.
\end{proof}

\subsection{Evalutation at Units }

\begin{proof}[\textbf{Theorem \textup{\ref{connection}}}][]
Employing the integral representation of $\mathscr{F}_1(x)$, given in Proposition \ref{int},
\begin{align}
\mathscr{F}_1(x)&=-\int_0^\infty\left(\frac{1}{1-e^{-t}}-\frac{1}{t}-\frac{1}{2}\right)\sum_{n=1}e^{-nxt},\nonumber
\end{align}
we see that
\begin{align*}
&2\mathscr{F}_1(2x)-3\mathscr{F}_1(x)+\mathscr{F}_1\left(\frac{x}{2}\right)\\
&=-\int_0^\infty\left(\frac{1}{1-e^{-t}}-\frac{1}{t}-\frac{1}{2}\right)\left\{\sum_{n\geq1}(1+(-1)^n)e^{-nxt}-3\sum_{n=1}e^{-nxt}+\sum_{n=1}e^{-nxt/2}\right\}dt\\
&=-\int_0^\infty\left(\frac{1}{1-e^{-t}}-\frac{1}{t}-\frac{1}{2}\right)\left\{\sum_{n\geq1}(-1)^ne^{-nxt}-2\sum_{n=1}e^{-nxt}+\sum_{n=1}e^{-nxt/2}\right\}dt\\
&=-\int_0^\infty\left(\frac{1}{1-e^{-t}}-\frac{1}{t}-\frac{1}{2}\right)\left\{\sum_{n\geq1}(-1)^ne^{-nxt}-\sum_{n=1}(1+(-1)^n)e^{-nxt/2}+\sum_{n=1}e^{-nxt/2}\right\}dt\\
&=-\int_0^\infty\left(\frac{1}{1-e^{-t}}-\frac{1}{t}-\frac{1}{2}\right)\left\{\sum_{n\geq1}(-1)^ne^{-nxt}-\sum_{n=1}(-1)^ne^{-nxt/2}\right\}dt.
\end{align*}
We now make change of the variable $t\to2t$ in the second integral so as to obtain
\begin{align*}
&2\mathscr{F}_1(2x)-3\mathscr{F}_1(x)+\mathscr{F}_1\left(\frac{x}{2}\right)\\
&=-\int_0^\infty\left(\frac{1}{1-e^{-t}}-\frac{1}{t}-\frac{1}{2}\right)\sum_{n\geq1}(-1)^ne^{-nxt}dt+\int_0^\infty\left(\frac{2}{1-e^{-2t}}-\frac{1}{t}-1\right)\sum_{n\geq1}(-1)^ne^{-nxt}dt\\
&=\int_0^\infty\left(\frac{2}{1-e^{-2t}}-\frac{1}{1-e^{-t}}-\frac{1}{2}\right)\sum_{n\geq1}(-1)^ne^{-nxt}dt\\
&=-\int_0^\infty\left(\frac{1}{1+e^{-t}}-\frac{1}{2}\right)\frac{1}{1+e^{xt}}dt,
\end{align*}
where in the last line we used the fact $\sum_{n\geq1}(-1)^ne^{-nxt}=-\frac{1}{1+e^{xt}}$. We now separate the integrals and use \eqref{j1x} so that
\begin{align*}
2\mathscr{F}_1(2x)-3\mathscr{F}_1(x)+\mathscr{F}_1\left(\frac{x}{2}\right)&=-J_1(x)+\frac{1}{2}\int_0^\infty\frac{1}{1+e^{xt}}dt\\
&=-J_1(x)+\frac{1}{2x}\log(2),
\end{align*}
which follows upon using the integral evaluation
\begin{align*}
\int_0^\infty\frac{1}{1+e^{xt}}dt=\frac{1}{x}\log(2). 
\end{align*}
This proves \eqref{connection eqn2}.

We next prove \eqref{J 2term}.
We first employ \eqref{connection eqn2} and then rearrange the terms to see that
\begin{align*}
&\mathcal{J}_1(x)-\frac{1}{x}\mathcal{J}_1\left(\frac{1}{x}\right)\\
&=-\left\{2\left(\mathcal{F}_1(2x)-\frac{1}{2x}\mathcal{F}_1\left(\frac{1}{2x}\right)\right)-3\left(\mathcal{F}_1(x)-\frac{1}{x}\mathcal{F}_1\left(\frac{1}{x}\right)\right)+\left(\mathcal{F}_1\left(\frac{x}{2}\right)-\frac{2}{x}\mathcal{F}_1\left(\frac{2}{x}\right)\right)\right\}.
\end{align*}
Equation \eqref{J 2term} now follows after invoking the two-term functional equation \eqref{refined 2-term} of the Ramanujan period function $\mathcal{F}_1(x)$.
\end{proof}

\begin{proof}[\textbf{Proposition \textup{\ref{J at naturals}}}][]
Employing the change of variable $e^{-t}=y$ in \eqref{j1x}, we have
\begin{align*}
\mathcal{J}_1(x)=\int_0^1\frac{y^{x-1}}{(1+y)(1+y^x)}dy+\log(2).
\end{align*}
Letting $x=1/n$ in the above equation and again making the change of variable $y=t^n$, thereby obtaining
\begin{align*}
    \mathcal{J}_1\left(\frac{1}{n}\right)=n\int_0^1\frac{1}{(1+t)(1+t^n)}dt+\log(2).
\end{align*}
Now employing the fact $\frac{n}{1+t^n}=\sum_{j=1}^n\frac{1}{1-te^{\pi i(2j-1)/n}}$, we see that
\begin{align*}
\mathcal{J}_1\left(\frac{1}{n}\right) =\sum_{j=1}^n\int_0^1\frac{1}{(1+t)(1-te^{\pi i(2j-1)/n})}dt+\log(2).
\end{align*}

We next make use of the simple integral evaluation $$\int_0^1\frac{1}{(1-at)(1+t)}dt=-\frac{1}{1+a}\log\left(\frac{1-a}{2}\right),\ a\neq1.$$
The result now follows upon using it in the case when $n$ is even as $e^{\pi i(2j-1)/n}\neq1$ for $1\leq j\leq n$. When $n$ is odd, we separate the term $j=(n+1)/2$ and then simplify to complete the proof of \eqref{J at naturals 1/n}. The other result \eqref{J at naturals n} follows directly from \eqref{J 2term}, which gives
\begin{align*}
\mathcal{J}_1\left(n\right)=\frac{1}{n}\mathcal{J}_1\left(\frac{1}{n}\right),
\end{align*}
 Now, simply substitute the value of $\mathcal{J}_1\left(\frac{1}{n}\right)$ from \eqref{J at naturals 1/n} into the above equation. 
\end{proof}

\begin{proof}[\textbf{Proposition \textup{\ref{J evaluation}}}][]
Define $\mathscr{B}_1$ and $\mathscr{B}_2$
to be the narrow classes with cycles $((2n))$ and $((n+1,2))$.
Note that we have \cite[p.~246]{raza}
\begin{align*}
\mathrm{Red}(\mathscr{B}_1)=\{u\},\qquad \mathrm{Red}(\mathscr{B}_2)=\left\{\frac{u+1}{2},\frac{2u}{u+1}\right\}.
\end{align*}
Let 
\begin{align*}
    \rho(\mathscr{B})=\sum_{k=1}^r\widetilde{P}(w_k,w_k'),
\end{align*}
where $\widetilde{P}(x,y)$ is defined in \eqref{p1 in terms of f}.

For $u=n+\sqrt{n^2-1}$, we have $u'=\frac{1}{u}$, thus 
\begin{align}
\rho(\mathscr{B}_1)=\widetilde{P}\left(u,\frac{1}{u}\right),\nonumber
\end{align}
and 
\begin{align}
\rho(\mathscr{B}_2)=\widetilde{P}\left(\frac{u+1}{2},\frac{u+1}{2u}\right)+\widetilde{P}\left(\frac{2u}{u+1},\frac{2}{u+1}\right).\nonumber
\end{align}
Consider
\begin{align}\label{roh1-rho2}
\rho(\mathscr{B}_1)-2\rho(\mathscr{B}_2)&=\widetilde{P}\left(u,\frac{1}{u}\right)-2\widetilde{P}\left(\frac{u+1}{2},\frac{u+1}{2u}\right)-2\widetilde{P}\left(\frac{2u}{u+1},\frac{2}{u+1}\right).
\end{align}
Equations \eqref{p1 in terms of f} and \eqref{roh1-rho2} together yield
\begin{align}
\rho(\mathscr{B}_1)-2\rho(\mathscr{B}_2)&=\mathcal{F}_1(u)-\mathcal{F}_1\left(\frac{1}{u}\right)-2\mathcal{F}_1\left(\frac{u+1}{2}\right)+2\mathcal{F}_1\left(\frac{u+1}{2u}\right)-2\mathcal{F}_1\left(\frac{2u}{u+1}\right)\nonumber\\
&\quad+2\mathcal{F}_1\left(\frac{2}{u+1}\right)-\frac{u^2-1}{2u}\left(\gamma+\log\left(\frac{u^2-1}{u}\right)\right)+\frac{2(u^2-1)}{(u+1)^2}\nonumber\\
&\quad\times\left(\gamma+\log\left(\frac{2(u^2-1)}{(u+1)^2}\right)\right)+\frac{u^2-1}{2u}\left(\gamma+\log\left(\frac{u^2-1}{2u}\right)\right).\nonumber
\end{align}
Invoking the two-term functional equation \eqref{refined 2-term} in the above equation
\begin{align}\label{before 3f-j1}
\rho(\mathscr{B}_1)-2\rho(\mathscr{B}_2)&=(1-u)\left\{\mathcal{F}_1(u)-\mathcal{F}_1\left(\frac{u+1}{2}\right)-\frac{2}{u+1}\mathcal{F}_1\left(\frac{2u}{u+1}\right)\right\} -\frac{u^2-1}{2u}\log\left(2\right)\nonumber\\
&\quad
+\frac{2(u-1)}{(u+1)}\left(\gamma+\log\left(\frac{2(u-1)}{(u+1)}\right)\right).
\end{align}
We rewrite the 5-term functional equation given in \eqref{5-term} as
\begin{align*}
&4\mathcal{F}_1(x)-2\mathcal{F}_1(2x)-\mathcal{F}_1\left(\frac{x}{2}\right)
-\mathcal{F}_1\left(\frac{x+1}{2}\right)-\frac{2}{x+1}\mathcal{F}_1\left(\frac{2x}{x+1}\right)  =0.
\end{align*}
This implies that
\begin{align*}
3\mathcal{F}_1(x)-2\mathcal{F}_1(2x)-\mathcal{F}_1\left(\frac{x}{2}\right)&=-\mathcal{F}_1(x)+\mathcal{F}_1\left(\frac{x+1}{2}\right)+\frac{2}{x+1}\mathcal{F}_1\left(\frac{2x}{x+1}\right).
\end{align*}
Now employing \eqref{connection eqn2} in the above equation, we obtain
\begin{align}\label{3f-j1}
&\mathcal{F}_1(x)-\mathcal{F}_1\left(\frac{x+1}{2}\right)-\frac{2}{x+1}\mathcal{F}_1\left(\frac{2x}{x+1}\right) =-\mathcal{J}_1(x) 
 +\frac{x+1}{2x}\log(2).
\end{align}
We now substitute the value from \eqref{3f-j1} into \eqref{before 3f-j1} to arrive at
\begin{align}
\rho(\mathscr{B}_1)-2\rho(\mathscr{B}_2)
&=(1-u)\left\{-\mathcal{J}_1(u)+\frac{u+1}{2u}\log(2)-\frac{2\gamma}{u+1}+\frac{u+1}{2u}\log(2)\right.\nonumber\\
&\left.\qquad -\frac{2}{u+1}\log\left(\frac{2(u-1)}{(u+1)}\right)\right\}\nonumber\\
&=(u-1)\left\{\mathcal{J}_1(u)+\frac{2\gamma}{u+1}-\frac{u^2+1}{u(u+1)}\log(2)
+\frac{2}{u+1}\log\left(\frac{u-1}{u+1}\right)\right\}.\nonumber
\end{align}
This along with \eqref{roh1-rho2} completes the proof of the proposition.
\end{proof}

\subsection{  Eisenstein series of weight 1  }
\begin{proof}[\textbf{Proposition \textup{\ref{Fourier}}}][]
For $\tau\in\mathbb{H}$, the definition of $\mathscr{F}_1(\tau)$, given in \eqref{ramanujan function}, implies that
\begin{align*}
\mathscr{F}_1(-\tau)- \mathscr{F}_1(\tau)
&=\sum_{n\geq1}\left(\psi(-n\tau)-\psi(n\tau)-\frac{1}{n\tau}+\log(n\tau)-\log(-n\tau)\right)\nonumber\\
&=\sum_{n\geq1}\left(\pi\cot(\pi n\tau)+i\pi\right),
\end{align*}
where we used the fact 
\begin{align*}
\psi( -\tau)-\psi(\tau) -\frac{1}{\tau}=\pi \cot(\pi \tau)\qquad (\tau\in\mathbb{H}),
\end{align*}
and $\log(\tau)-\log(-\tau)=i\pi.$ 
We now see that
\begin{align}\label{laststep}
\mathscr{F}_1(-\tau)-\mathscr{F}_1(\tau) &=-\sum_{n\geq1}\left(\pi\cot(-\pi n\tau)-i\pi\right)\nonumber\\
&=-\sum_{n\geq1}\left(i\pi\frac{e^{-\pi in\tau}+e^{\pi in\tau}}{e^{-\pi in\tau}-e^{\pi in\tau}}-i\pi\right) =-\sum_{n\geq1}\left(\frac{2i\pi e^{\pi in\tau}}{e^{-\pi in\tau}-e^{\pi in\tau}}\right)\nonumber\\
&=-2\pi i\sum_{n\geq1}\frac{1}{e^{-2\pi in\tau}-1} =-2\pi i\sum_{n\geq1}\sum_{m\geq1}e^{2\pi imn\tau}.
\end{align}
The proposition now follows by substituting $$\mathscr{F}_1(\tau)=\mathcal{F}_1(\tau)+\frac{1}{2}\left(\gamma-\log\left(\frac{2\pi}{\tau}\right)\right),$$ as given in \eqref{refined ramanujan function}, into \eqref{laststep} and simplifying the expressions.
\end{proof}

\subsection{As a limit of a period function of Lewis-Zagier}

\begin{proof}[\textbf{Proposition \textup{\ref{properties of GRPF}}}][]
(1) Assume Re$(s)>1$ and using the definition of the Hurwitz zeta function,
\begin{align*}
\zeta(s,x)=\sum_{m=0}^\infty\frac{1}{(m+x)^s},
\end{align*}
 we have
\begin{align}
&\sum_{n=1}^\infty\left(\zeta(2s,nx)-\frac{1}{2}(nx)^{-2s}+\frac{1}{1-2s}(nx)^{1-2s}\right)+\frac{1}{2}\zeta(2s)+\frac{x^{1-2s}}{2s-1}\zeta(2s-1)\nonumber\\
&=\sum_{n=1}^\infty\sum_{m=0}^\infty\frac{1}{(m+nx)^{2s}}-\frac{1}{2}x^{-2s}\zeta(2s)+\frac{1}{2}\zeta(2s)\nonumber\\
&=\sum_{n=1}^\infty\sum_{m=1}^\infty\frac{1}{(m+nx)^{2s}}+\frac{1}{2}x^{-2s}\zeta(2s)+\frac{1}{2}\zeta(2s)\nonumber\\
&={\sum_{m,n\geq0}}^*\frac{1}{(mx+n)^{2s}}, \nonumber
\end{align}
which is nothing but the function $\psi_s^+(x)$ defined in \eqref{psi+}. This proves the first part of the proposition.

\noindent
(2) For Re$(s)>-1,\ s\neq1$ and Re$(a)>0$,  we have \cite[p.~609, Formula 25.11.27]{nist}
\begin{align}\label{irh}
\zeta(s,a)-\frac{1}{2}a^{-s}-\frac{a^{1-s}}{s-1}=\frac{1}{\Gamma(s)}\int_0^\infty\left(\frac{1}{e^t-1}-\frac{1}{t}+\frac{1}{2}\right)e^{-at}t^{s-1}dt.
\end{align}
Replacing $s$ by $2s$ and letting $a=nx$ in \eqref{irh} and then taking sum over $n$, we see that
\begin{align}\label{1}
\sum_{n=1}^\infty\left(\zeta(2s,nx)-\frac{1}{2}(nx)^{-2s}+\frac{(nx)^{1-2s}}{1-2s}\right)=\frac{1}{\Gamma(2s)}\int_0^\infty\left(\frac{1}{e^t-1}-\frac{1}{t}+\frac{1}{2}\right)\frac{t^{2s-1}}{e^{xt}-1}dt.
\end{align}
Using the well-known integral representation of $\zeta(s)$ \cite[p.~604, Formula 25.5.1]{nist}
\begin{align*}
\zeta(s)=\frac{1}{\Gamma(s)}\int_0^\infty\frac{t^{s-1}}{e^t-1}dt, \qquad(\mathrm{Re}(s)>1)
\end{align*}
one can obtain
\begin{align}\label{2}
\frac{1}{2}\zeta(2s)=\frac{x^{2s}}{2}\frac{1}{\Gamma(2s)}\int_0^\infty\frac{t^{2s-1}}{e^{xt}-1}dt,
\end{align}
and
\begin{align}\label{3}
\frac{1}{2s-1}x^{1-2s}\zeta(2s-1)=\frac{1}{\Gamma(2s)}\int_0^\infty\frac{t^{2s-1}}{t(e^{xt}-1)}dt.
\end{align}
Substituting values from \eqref{1}, \eqref{2} and \eqref{3} in \eqref{f-psi}, we complete the proof of the second part.

\noindent
(3) Adding and subtracting $$\frac{1}{\Gamma(2s)}\sum_{k=1}^m\frac{B_{2k}\Gamma(2s+2k-1)}{(2k)!(nx)^{2s+2k-1}}$$ $B_k$'s being Bernoulli numbers, in the summand of the series on the right-hand side of \eqref{f-psi} so as to obtain
\begin{align}\label{asub}
\psi_s^+(x)&=\sum_{n=1}^\infty\left(\zeta(2s,nx)-\frac{1}{2}(nx)^{-2s}+\frac{1}{1-2s}(nx)^{1-2s}-\frac{1}{\Gamma(2s)}\sum_{k=1}^m\frac{B_{2k}\Gamma(2s+2k-1)}{(2k)!x^{2s+2k-1}}\right)\nonumber\\
&\quad+\frac{1}{2}\zeta(2s)+\frac{x^{1-2s}}{2s-1}\zeta(2s-1)+\frac{1}{\Gamma(2s)}\sum_{k=1}^m\frac{B_{2k}\Gamma(2s+2k-1)\zeta(2s+2k-1)}{(2k)!x^{2s+2k-1}}.
\end{align}

As $a\to\infty$ in the sector $|\arg(a)|\leq\pi-\delta$, we have \cite[p.~610, Formula 25.11.43]{nist}
\begin{align*}
\zeta(s,a)-\frac{1}{2}a^{-s}-\frac{a^{1-s}}{s-1}=\frac{1}{\Gamma(s)}\sum_{k=1}^m\frac{B_{2k}\Gamma(s+2k-1)}{(2k)!a^{s+2k-1}}+\mathcal{O}\left(\frac{1}{x^{s+2m+1}}\right),
\end{align*}
for $m\geq0$ be any integer.

Using the above asymptotic, it is easy to see that the infinite series in \eqref{asub} defines an analytic function of $s$ in the region Re$(s)>-m$. Moreover, observe that the term $\frac{x^{1-2s}}{2s-1}\zeta(2s-1)$ has only a simple pole at $s=1$, and the point $s=1/2$ is the removable singularity of the term $\frac{1}{2}\zeta(2s)+\frac{x^{1-2s}}{2s-1}\zeta(2s-1)$. Whereas, we can also see that the finite sum is an entire function of $s$ because we can re-write 
\begin{align*}
\frac{\Gamma(2s+2k-1)\zeta(2s+2k-1)}{\Gamma(2s)}&=(2s+2k-2)\zeta(2s+2k-1)\frac{\Gamma(2s+2k-2)}{\Gamma(2s)}\\
&=\left\{(2s+2k-2)\zeta(2s+2k-1)\right\}\left\{(2s+2k-3)\cdots(2s-1)(2s)\right\}.
\end{align*}
Hence this term is also an entire function of $s$. Consequently, we can conclude that the right-hand side of \eqref{asub} provides the analytic continuation of the function $\psi_s^+(x)$ in the region Re$(s)>-m$, with only a simple pole at $s=1$. Moreover, the residue at $s=1$ comes from the term $\frac{x^{1-2s}}{2s-1}\zeta(2s-1)$,  which can be evaluated to $1/(2x)$. 

The above discussion is valid for any integer $m\geq0$. Therefore, it  gives the desired analytic continuation in the whole complex plane and completes the proof of our aim.

\noindent
(4) Letting $s\to1/2$ in \eqref{f-psi}, and using the well-known result
\begin{align*}
\lim_{s\to1/2}\left\{\zeta(2s,x)-\frac{x^{1-2s}}{2s-1}\right\}=\log(x)-\psi(x),
\end{align*}
along with the fact that 
\begin{align*}
\lim_{s\to1/2}\left\{\frac{1}{2}\zeta(2s)+\frac{x^{1-2s}}{2s-1}\zeta(2s-1)\right\}=\frac{1}{2}\left(\gamma-\log\left(\frac{2\pi}{x}\right)\right),
\end{align*}
and the definition of the Ramanujan period function \eqref{refined ramanujan function}, we complete the proof of our claim.
\end{proof}

\begin{proof}[\textbf{Theorem \textup{\ref{hecke operator functional equationss}}}][]
The proof of this result runs along the same lines as that of Theorem \ref{hecke operator functional equations}, so we will be concise here.

For Re$(s)>1$ and $x>0$, let us define 
\begin{align}
B(x,s):=\frac{1}{\Gamma(2s)}\int_0^\infty\frac{t^{2s-1}}{\left(e^{xt}-1\right)\left(e^{t}-1\right)}dt.\nonumber
\end{align}
Note that we can rewrite the integral above as
\begin{align}
B(x,s)&=\frac{x^{1-2s}}{2s-1}\zeta(2s-1)-\frac{1}{2}x^{-2s}\zeta(2s)+\frac{1}{\Gamma(2s)}\int_0^\infty\left(\frac{1}{e^{t}-1}-\frac{1}{t}+\frac{1}{2}\right)\frac{t^{2s-1}}{e^{xt}-1}dt.\nonumber
\end{align}
Observe that the right-hand side is analytic for Re$(s)>0$ except at $s=1/2$ and $s=1$ where it has simple poles. Therefore, the expression on the right-hand of the above equation gives analytic continuation of the function $B(x,s)$ in the region Re$(s)>0$. 

Using the integral representation for $\psi_s^+(x)$ from \eqref{integral representation for GRPF}, we have
\begin{align}\label{final B}
B(x,s)&=-\frac{1}{2}\zeta(2s)-\frac{1}{2}x^{-2s}\zeta(2s)+\psi_s^+(x).
\end{align}
It is easy to see that
\begin{align}
&-\frac{(2\pi)^{2s}}{4\Gamma(2s)}\int_0^\infty\mathscr{C}(ixt,iyt)t^{2s-1}dt=\frac{1}{y^{2s}}B\left(\frac{x}{y},s\right)+\frac{1}{2}\zeta(2s)\left(\frac{1}{x^{2s}}+\frac{1}{y^{2s}}\right).\nonumber
\end{align}
Invoking \eqref{final B} in the above equation,   we are led to
\begin{align}\label{eval in terms of F}
-\frac{(2\pi)^{2s}}{4\Gamma(2s)}\int_0^\infty\mathscr{C}(ixt,iyt)t^{2s-1}dt=\frac{1}{y^{2s}}\psi_s^+\left(\frac{x}{y}\right).
\end{align}
We apply $\widetilde{T}_n$ on both sides of the above equation. 

By repeating the calculations performed after (8.12), but now for a general $s$, one can complete the proof of the result.
\end{proof}

\begin{proof}[\textbf{Proposition \textup{\ref{RaBC}}}][]
Using the functional equation of the Riemann zeta function 
\begin{align*}
\zeta(s)=2^s\pi^{s-1}\sin\left(\frac{\pi s}{2}\right)\Gamma(1-s)\zeta(1-s)
\end{align*}
in \eqref{eisen}, we obtain
\begin{align}
E_{2s}(\tau)&=1+\frac{(2\pi)^{2s}}{\cos(\pi s)\Gamma(2s)\zeta(2s)}\sum_{n=1}^\infty\sigma_{2s-1}(n)q^n\nonumber\\
&=1+\frac{2(2\pi)^{2s}e^{-\pi is}}{\Gamma(2s)\zeta(2s)(1+e^{-2\pi is})}\sum_{n=1}^\infty\sigma_{2s-1}(n)q^n.\nonumber
\end{align}
This implies that
\begin{align}\label{es final1}
\frac{1}{2}\left(1+e^{-2\pi is}\right)\zeta(2s)E_{2s}(\tau)&=\frac{1}{2}\left(1+e^{-2\pi is}\right)\zeta(2s)+\frac{(-2\pi i)^{2s}}{\Gamma(2s)}\sum_{n=1}^\infty\sigma_{2s-1}(n)q^n.
\end{align}

Lewis and Zagier defined the following function \cite[p.~243, (4.3)]{lz}
\begin{align}\label{ftau}
f_s(\tau):=\psi^+_s(\tau)+\tau^{-2s}\psi^+_s(-1/\tau),
\end{align}
where $\psi^+_s(\tau)$ is defined in \eqref{psi+}. They showed that \cite[p.~243, (4.3)]{lz}
\begin{align}\label{ftau1}
f_s(\tau)=\frac{1}{2}\left(1+e^{-2\pi is}\right)\zeta(2s)+\frac{(-2\pi i)^{2s}}{\Gamma(2s)}\sum_{n=1}^\infty\sigma_{2s-1}(n)q^n.
\end{align}
From \eqref{es final1} and \eqref{ftau1}, we have
\begin{align}\label{es final}
\frac{1}{2}\left(1+e^{-2\pi is}\right)\zeta(2s)E_{2s}(\tau)&=f_s(\tau).
\end{align}
Then \eqref{es final} implies that
\begin{align}\label{5.6}
f_s(\tau)-\frac{1}{\tau^{2s}}f_s\left(-\frac{1}{\tau}\right)&=\frac{1}{2}\left(1+e^{-2\pi is}\right)\zeta(2s)\left\{E_{2s}(\tau)-\frac{1}{\tau^{2s}}E_{2s}\left(-\frac{1}{\tau}\right)\right\}.
\end{align}
On the other hand, \eqref{ftau} results in
\begin{align}\label{5.5}
f_s(\tau)-\frac{1}{\tau^{2s}}f_s\left(-\frac{1}{\tau}\right)&=\psi^+_s(\tau)+\tau^{-2s}\psi^+_s(-1/\tau)-\frac{1}{\tau^{2s}}\left\{\psi^+_s(-1/\tau)+(-\tau)^{2s}\psi^+_s(\tau)\right\}\nonumber\\
&=\psi^+_s(\tau)\left(1-(-1)^{2s}\right)\nonumber\\
&=\psi^+_s(\tau)\left(1-e^{-2\pi is}\right).
\end{align}
Equations \eqref{5.5} and \eqref{5.6} together yield
\begin{align}
\psi^+_s(\tau)&=\frac{1}{2}\frac{\left(1+e^{-2\pi is}\right)}{(1-e^{-2\pi is})}\zeta(2s)\left\{E_{2s}(\tau)-\frac{1}{\tau^{2s}}E_{2s}\left(-\frac{1}{\tau}\right)\right\}\nonumber\\
&=\frac{1}{2}\frac{\left(1+e^{-2\pi is}\right)}{(1-e^{-2\pi is})}\zeta(2s)\Psi_s(\tau),\nonumber
\end{align}
where $\Psi_s(\tau)$ is defined in \eqref{psi+}. This completes the proof of the proposition.
\end{proof}

\bigskip
{\bf{Acknowledgements:}} 
We would like to thank the referees for their careful reading of the manuscript and for their numerous helpful comments and suggestions which greatly improved the exposition of this paper. 
The first author was partially supported by   NRF$2022R1A2B5B- 0100187113$, BSRI-NRF$2021R1A6A1A10042944$ and the second author was partially supported by the FIG grant of IIT
Roorkee. Both the authors thank respective funding agencies/institutes.






\begin{thebibliography}{00}

\bibitem{as} 
M.~Abramowitz and I.~A.~Stegun, \emph{Handbook of Mathematical Functions, with Formulas, Graphs, and Mathematical Tables}, 9th edition, Dover Publications, New York, 1970.

%

\bibitem{andrewsberndt}
G.~E.~Andrews, and B.~C.~Berndt, \emph{Ramanujan's lost notebook. Vol. IV}, New York: Springer, 2005.

\bibitem{BD}
B.~Berndt and A.~Dixit, \emph{A transformation formula involving the Gamma and Riemann zeta functions in Ramanujan's Lost notebook} In: Alladi, K., Klauder, J., Rao, C.R. (eds.) The Legacy of Alladi Ramakrishnan in the Mathematical Sciences, pp. 199--210. Springer, New York (2010).

\bibitem{bc}
S.~Bettin, J.~B.~Conrey, \emph{Period functions and cotangent sums}, Algebra Number Theory \textbf{7}(1), 215--242 (2013).

\bibitem{choiekumar}
Y.~Choie and R.~Kumar, \emph{Arithmetic properties of Herglotz-Zagier-Novikov function}, Adv. in  Math. \textbf{433} (2023), Paper No. 109315.



\bibitem{choiezagier}
Y.~Choie and D.~ Zagier, \emph{Rational period functions for PSL$(2, \mathbb{Z})$}, Contemp. Math. 143, American Mathematical Society, Providence, 1993, 89--108.


\bibitem{cohen}
H.~Cohen, \emph{Number Theory-Volume II: Analytic and Modern Tools}, Graduate Texts in Mathematics Vol. 240: Springer-Verlag: New York 2007.

\bibitem{dixitijnt}
A.~Dixit, \emph{Analogues of a transformation formula of Ramanujan}, Int. J. Number Theory \textbf{7}(5) (2011) 1151--1172.




\bibitem{her}
G.~Herglotz, \emph{\"{U}ber die Kroneckersche Grenzformel f\"{u}r reelle, quadratische K\"{o}rper I}, Ber.~Verhandl.~S\"{a}chsischen~Akad.~Wiss.~Leipzig~\textbf{75} (1923), 3--14.




\bibitem{lz}
J.~Lewis and D.~Zagier, \emph{Period functions for Maass wave forms. I.}, 
Annals of Mathematics,  153  (2001),  no.1, 191--258. 

\bibitem{kuru}
N.~Kurukawa, \emph{Automorphy of the principal Eisenstein series of weight $1$: an application of the double sine function}, Kodai Math. J. \textbf{32} (2009), 391--403. 

\bibitem{muzwil}
H.~Muzaffar and K.~S.~Williams, \emph{A restricted Epstein zeta function and the evaluation of some definite integrals}, Acta Arith.~\textbf{104}, no. 1 (2002), 23--66.


\bibitem{nist}
F.~W.~J.~Olver, D.~W.~Lozier, R.~F.~Boisvert, C.~W.~Clark (eds.), \emph{NIST Handbook of Mathematical Functions} Cambridge University Press, Cambridge (2010).

\bibitem{raza}
D.~Radchenko and D.~Zagier, \emph{Arithmetic properties of the Herglotz Function}, J. reine angew. Math. \textbf{797} (2023), 229--253.

\bibitem{rlnb} 
S.~Ramanujan, \emph{The Lost Notebook and Other Unpublished Papers}, Narosa, New Delhi (1988).


\bibitem{titch1}
E.~C.~Titchmarsh, \emph{Theory of Fourier Integrals}, 2nd edn. Clarendon Press, Oxford, 1948.

\bibitem{titch}
E.~C.~Titchmarsh, \emph{The Theory of the Riemann Zeta Function}, 2nd ed., Revised by D.~R.~Heath-

\bibitem{vz}
M.~Vlasenko and D.~Zagier, \emph{Higher Kronecker ``limit" formulas for real quadratic fields}, J. reine angew. Math. 679, pp. 23--64 (2013).

\bibitem{ww}
E.~T.~Whittaker and G.~N.~Watson, \emph{A course of Modern Analysis}, Cambridge Mathematical Library, Cambridge University Press, Cambridge, 1996.



\bibitem{zagier1975}
D.~Zagier, \emph{A Kronecker limit formula for real quadratic fields}, Math.~Ann.~\textbf{213} (1975), 153--184.


\end{thebibliography}
\end{document}